\documentclass{article}

\usepackage{amssymb,stmaryrd}
\newcommand{\res}{{\upharpoonright}}
\newcommand{\NN}{\mathbb{N}}
\newcommand{\QQ}{\mathbb{Q}}
\newcommand{\cat}{{^\smallfrown}}
\newcommand{\llb}{\llbracket}
\newcommand{\rrb}{\rrbracket}
\newcommand{\liff}{\Leftrightarrow}
\newcommand{\limp}{\Rightarrow}
\newcommand{\wt}{\mathrm{wt}}
\newcommand{\dwt}{\mathrm{dwt}}
\newcommand{\pwt}{\mathrm{pwt}}
\newcommand{\vwt}{\mathrm{vwt}}
\newcommand{\effdim}{\mathrm{effdim}}
\newcommand{\PA}{\mathsf{PA}}
\newcommand{\WKLo}{\mathsf{WKL}_0}
\newcommand{\DNR}{\mathrm{DNR}}
\newcommand{\KA}{\mathrm{KA}}
\newcommand{\KP}{\mathrm{KP}}
\newcommand{\KS}{\mathrm{KS}}
\newcommand{\BC}{\mathrm{BC}}
\newcommand{\LR}{\mathrm{LR}}
\newcommand{\leT}{\leq_{\mathrm{T}}}
\newcommand{\leLR}{\leq_\LR}
\newcommand{\nleLR}{\nleq_\LR}
\newcommand{\Ybar}{{\overline{Y}}}

\usepackage{amsthm}
\theoremstyle{definition}
\newtheorem{thm}{Theorem}[section]
\newtheorem{lem}[thm]{Lemma}
\newtheorem{cor}[thm]{Corollary}
\newtheorem{dfn}[thm]{Definition}
\newtheorem{rem}[thm]{Remark}

\usepackage[backref=page,colorlinks=true,allcolors=blue]{hyperref}

\begin{document}

\begin{center}
  {\LARGE Propagation of partial randomness}\\[14pt]
  {\large Kojiro Higuchi}\\
  Department of Mathematics and Informatics\\
  Faculty of Science, Chiba University\\
  1-33 Yayoi-cho, Inage, Chiba, JAPAN\\
  \href{mailto:khiguchi@g.math.s.chiba-u.ac.jp}{khiguchi@g.math.s.chiba-u.ac.jp}\\[9pt]
  {\large W. M. Phillip Hudelson}\\
  One Oxford Center, Suite 2600\\
  301 Grant Street\\
  Pittsburgh, PA 15219, USA\\
  \href{mailto:phil.hudelson@gmail.com}{phil.hudelson@gmail.com}\\[9pt]
  {\large Stephen G. Simpson}\\
  Department of Mathematics\\
  Pennsylvania State University\\
  University Park, PA 16802, USA\\
  \href{http://www.personal.psu.edu/t20}{http://www.personal.psu.edu/t20}\\
  \href{mailto:t20@psu.edu}{t20@psu.edu}\\[9pt]
  {\large Keita Yokoyama}\\
  School of Information Science\\
  Japan Advanced Institute of Science and Technology\\
  1-1 Asahidai, Nomi, Ishikawa, 923-1292, JAPAN\\
  \href{mailto:y-keita@jaist.ac.jp}{y-keita@jaist.ac.jp}\\[12pt]
  First draft: April 21, 2011\\
  This draft: \today
\end{center}

\addcontentsline{toc}{section}{Abstract}

\begin{abstract}
  Let $f$ be a computable function from finite sequences of $0$'s and
  $1$'s to real numbers.  We prove that strong $f$-randomness implies
  strong $f$-randomness relative to a $\PA$-degree.  We also prove: if
  $X$ is strongly $f$-random and Turing reducible to $Y$ where $Y$ is
  Martin-L\"of random relative to $Z$, then $X$ is strongly $f$-random
  relative to $Z$.  In addition, we prove analogous propagation
  results for other notions of partial randomness, including
  non-K-triviality and autocomplexity.  We prove that $f$-randomness
  relative to a $\PA$-degree implies strong $f$-randomness, hence
  $f$-randomness does not imply $f$-randomness relative to a
  $\PA$-degree.
\end{abstract}

\medskip

\noindent Keywords: partial randomness, effective Hausdorff dimension,
Martin-L\"of randomness, Kolmogorov complexity, models of arithmetic.

\medskip

\noindent 2010 Mathematics Subject Classification: Primary 03D32,
Secondary 03D28, 68Q30, 03H15, 03C62, 03F30.

\medskip

\noindent A version of this paper will appear in \emph{Annals of Pure
  and Applied Logic}.

\newpage

\tableofcontents

\section{Introduction}
\label{sec:intro}

We begin by recalling two known results concerning Martin-L\"of
randomness relative to a Turing oracle.  Let $\NN$ denote the set of
positive integers.  Let $\{0,1\}^\NN$ denote the \underline{Cantor
  space}, i.e., the set of infinite sequences of $0$'s and $1$'s.

\begin{thm}
  \label{thm:mlrXYZ}
  Let $X\in\{0,1\}^\NN$ be Martin-L\"of random.  Suppose $X$ is Turing
  reducible to $Y$ where $Y$ is Martin-L\"of random relative to $Z$.
  Then $X$ is Martin-L\"of random relative to $Z$.
\end{thm}

\begin{thm}
  \label{thm:mlrQ}
  Let $Q$ be a nonempty $\Pi^0_1$ subset of $\{0,1\}^\NN$.  If
  $X\in\{0,1\}^\NN$ is Martin-L\"of random, then $X$ is Martin-L\"of
  random relative to some $Z\in Q$.
\end{thm}

Recall from \cite{js} that a \underline{$\PA$-degree} is defined to be
the Turing degree of a complete consistent theory extending
first-order Peano arithmetic.  It is well known that, via G\"odel
numbering, the set of complete consistent theories extending
first-order Peano arithmetic may be viewed as a $\Pi^0_1$ subset of
$\{0,1\}^\NN$.  Moreover
\nocite{rec-fcn-thy,rm2001}\cite{js,scott-binum,massrand,pizowkl},
this particular $\Pi^0_1$ subset of $\{0,1\}^\NN$ is \emph{universal}
in the following sense: a Turing oracle $Z$ is of $\PA$-degree if and
only if every nonempty $\Pi^0_1$ subset of $\{0,1\}^\NN$ contains an
element which is Turing reducible to $Z$.  Consequently, Theorem
\ref{thm:mlrQ} may be restated as follows:
\begin{thm}
  \label{thm:mlrPA}
  If $X\in\{0,1\}^\NN$ is Martin-L\"of random, then $X$ is
  Martin-L\"of random relative to some $\PA$-degree.
\end{thm}

Theorem \ref{thm:mlrXYZ}, which we call the \emph{XYZ Theorem}, is due
to Miller/Yu \cite[Theorem 4.3]{mi-yu-1}.  Theorems \ref{thm:mlrQ} and
\ref{thm:mlrPA} are due independently to several groups of
researchers: Downey et al \cite[Proposition 7.4]{do-hi-mi-ni},
Reimann/Slaman \cite[Theorem 4.5]{re-sl-measures} (also cited in
\cite{do-hi-mi-ni}), and Simpson/Yokoyama \cite[Lemma 3.3]{nswwkl}.

Theorems \ref{thm:mlrQ} and \ref{thm:mlrPA} have been very useful in
the study of randomness.  Reimann/Slaman \cite{re-sl-measures} used
Theorem \ref{thm:mlrQ} to prove that any noncomputable
$X\in\{0,1\}^\NN$ is nonatomically random with respect to some
probability measure on $\{0,1\}^\NN$.  Simpson/Yokoyama \cite{nswwkl}
used a generalization of Theorem \ref{thm:mlrQ} to study the reverse
mathematics of Loeb measure.  Recently Brattka/Miller/Nies
\cite{br-mi-ni} used Theorem \ref{thm:mlrQ} to prove that $x\in[0,1]$
is random if and only if every computable continuous function of
bounded variation is differentiable at $x$.

The theme of Theorems \ref{thm:mlrXYZ}--\ref{thm:mlrPA} is what might
be called ``propagation of Martin-L\"of randomness.''  Namely, all of
these theorems assert that if $X$ is Martin-L\"of random then $X$ is
Martin-L\"of random relative to certain Turing oracles.

The purpose of this paper is to present some new results which are
generalizations of Theorems \ref{thm:mlrXYZ}--\ref{thm:mlrPA}.  The
theme of our new results might be called ``propagation of partial
randomness.''  Here ``partial randomness'' refers to certain
properties which are in the same vein as Martin-L\"of randomness.
Recent studies of partial randomness include \nocite{9asian}
\cite{ca-st-te,hudelson-complex,kms-tams,miller-dim,reimann-phd,re-st-tests,tadaki-A}.
Our main new results involve a specific notion of partial randomness
known as \emph{strong $f$-randomness} where $f$ is an arbitrary
computable function from finite sequences of $0$'s and $1$'s to real
numbers.  Along the way we present some old and new characterizations
of strong $f$-randomness.  We also consider other notions of partial
randomness including \emph{complexity}
\cite{binns-complex,hi-ki-null,kms-tams}, \emph{autocomplexity}
\cite{kms-tams}, and \emph{non-K-triviality}
\cite{do-hi-book,nies-book}.

The plan of this paper is as follows.  In \S\ref{sec:frsfr} we define
\emph{$f$-randomness} and \emph{strong $f$-randomness} and
characterize these notions in terms of Kolmogorov complexity.  In
\S\ref{sec:grsfr} we prove that $(f+2\log_2f)$-randomness implies
strong $f$-randomness.  Note that \S\ref{sec:frsfr} and
\S\ref{sec:grsfr} and \S\ref{sec:vfr} are largely expository.  In
\S\ref{sec:sfr} we present our main results concerning propagation of
strong $f$-randomness.  Namely, we prove appropriate generalizations
of Theorems \ref{thm:mlrXYZ}--\ref{thm:mlrPA} with Martin-L\"of
randomness replaced by strong $f$-randomness.  In \S\ref{sec:nkt} and
\S\ref{sec:dnr} we prove analogous results concerning propagation of
non-K-triviality and propagation of diagonal nonrecursiveness,
respectively.  In \S\ref{sec:auto} we prove analogous results
concerning propagation of autocomplexity, and we characterize
autocomplexity in terms of $f$-randomness and strong $f$-randomness.
In \S\ref{sec:vfr} we define \emph{vehement $f$-randomness} and prove
that it is equivalent to strong $f$-randomness, provided $f$ is
\emph{convex}.  In \S\ref{sec:vfrPA} we prove a version of Theorem
\ref{thm:mlrQ} with Martin-L\"of randomness replaced by vehement
$f$-randomness.  In \S\ref{sec:sfrPA} we present two new
characterizations of strong $f$-randomness.  In \S\ref{sec:non-pfr} we
show that our results concerning propagation of strong $f$-randomness
fail for $f$-randomness.

\section{$f$-randomness and strong $f$-randomness}
\label{sec:frsfr}

Let $f:\{0,1\}^*\to[-\infty,\infty]$ be an arbitrary computable
function from finite sequences of $0$'s and $1$'s to the
extended\footnote{We define $f:\{0,1\}^*\to[-\infty,\infty]$ to be
  computable if $f/(|f|+1):\{0,1\}^*\to[-1,1]$ is computable.} real
numbers.  In this section we define what it means for
$X\in\{0,1\}^\NN$ to be $f$-random, strongly $f$-random, $f$-complex,
and strongly $f$-complex.

Recall that according to Schnorr's Theorem (see \cite[Theorem
6.2.3]{do-hi-book} or \cite[Theorem 3.2.9]{nies-book} or \cite[Theorem
10.7]{aedsh}), $X$ is Martin-L\"of random if and only if for all $n$
the prefix-free Kolmogorov complexity of the first $n$ bits of $X$ is
at least $n$ modulo an additive constant.  In this section we prove
generalizations of Schnorr's Theorem, replacing Martin-L\"of
randomness by $f$-randomness and strong $f$-randomness.  Our proofs
are modeled on one of the standard proofs \cite[Theorem 10.7]{aedsh}
of Schnorr's Theorem.

This section is mostly expository.  For the history of the concepts
and results in this section, see Calude/Staiger/Terwijn
\cite{ca-st-te} and Tadaki \cite{tadaki-A}.

\begin{dfn}
  \label{dfn:fcsfc}
  For $X\in\{0,1\}^\NN$ and $n\in\NN$ we write $X\res
  n=X\res\{1,\ldots,n\}=$ the first $n$ bits of $X$.  Given
  $f:\{0,1\}^*\to[-\infty,\infty]$ we define $X$ to be
  \underline{$f$-complex} or \underline{strongly $f$-complex} if
  \begin{center}
    $\exists c\,\forall n\,(\KP(X\res n)\ge f(X\res n)-c)$
  \end{center}
  or
  \begin{center}
    $\exists c\,\forall n\,(\KA(X\res n)\ge f(X\res n)-c)$
  \end{center}
  respectively.  Here $\KP$ and $\KA$ denote \underline{prefix-free
    complexity} (see \cite[\S3.5]{do-hi-book} or
  \cite[\S2.2]{nies-book} or \cite[\S10]{aedsh}) and \underline{a
    priori complexity} (see \cite[\S3.16]{do-hi-book} or \cite{us-sh})
  respectively.
\end{dfn}

\begin{dfn}
  \label{dfn:dwtpwt} 
  Given $f:\{0,1\}^*\to[-\infty,\infty]$, the \underline{$f$-weight}
  of $\sigma\in\{0,1\}^*$ is defined as
  $\wt_f(\sigma)=2^{-f(\sigma)}$.  The \underline{direct $f$-weight}
  of $A\subseteq\{0,1\}^*$ is defined as $\dwt_f(A)=\sum_{\sigma\in
    A}\wt_f(\sigma)$.  A set $P\subseteq\{0,1\}^*$ is said to be
  \underline{prefix-free} if no element of $P$ is a proper initial
  segment of an element of $P$.  The \underline{prefix-free
    $f$-weight} of $A$ is defined as
  \begin{center}
    $\pwt_f(A)=\sup\{\dwt_f(P)\mid P\subseteq A$ is prefix-free$\}$.
  \end{center}
\end{dfn}

\begin{dfn}
  For $\sigma\in\{0,1\}^*$ we write
  $\llb\sigma\rrb=\{X\in\{0,1\}^\NN\mid\sigma\subset X\}$.  For
  $A\subseteq\{0,1\}^*$ we write $\llb A\rrb=\bigcup_{\sigma\in
    A}\llb\sigma\rrb$ and
  \begin{center}
    $\widehat{A}=\{\sigma\in
    A\mid\nexists\rho\,(\rho\subset\sigma$ and $\rho\in A)\}=$ the
    set of minimal elements of $A$.
  \end{center}
  Note that $\widehat{A}$ is prefix-free and
  $\llb\widehat{A}\rrb=\llb A\rrb$.
\end{dfn}

We write \underline{r.e.} as an abbreviation for \emph{recursively
  enumerable}.  A sequence of sets $A_i\subseteq\{0,1\}^*$, $i\in\NN$
is said to be \underline{uniformly r.e.} if $\{(\sigma,i)\mid\sigma\in
A_i\}$ is r.e.

\begin{dfn}
  \label{dfn:frsfr}
  Assume that $f:\{0,1\}^*\to[-\infty,\infty]$ is computable.  We
  define $X\in\{0,1\}^\NN$ to be \underline{$f$-random} or
  \underline{strongly $f$-random} if $X\notin\bigcap_i\llb A_i\rrb$
  whenever $A_i$ is uniformly r.e.\ with $\dwt_f(A_i)\le2^{-i}$ or
  $\pwt_f(A_i)\le2^{-i}$ respectively.
\end{dfn}

\begin{rem}
  Since $\pwt_f(A)\le\dwt_f(A)$ for all $A$, it is clear that strong
  $f$-randomness implies $f$-randomness.  Similarly, since $\exists
  c\,\forall\tau\,(\KA(\tau)\le\KP(\tau)+c)$, it is clear that strong
  $f$-complexity implies $f$-complexity.  Note also that $\wt_f$ is a
  \emph{premeasure} in the sense of \cite[Definition 1]{reimann-eff}.
\end{rem}

The next theorem is a straightforward generalization of Tadaki
\cite[Theorem 3.1]{tadaki-A}.

\begin{thm}
  \label{thm:rc}
  Let $f:\{0,1\}^*\to[-\infty,\infty]$ be computable.  Then
  $f$-randomness is equivalent to $f$-complexity.
\end{thm}

\begin{proof}
  Suppose $X$ is $f$-random.  Let
  $S_i=\{\tau\mid\KP(\tau)<f(\tau)-i\}$.  Clearly $S_i$ is uniformly
  r.e., and by Kraft's Inequality \cite[Theorem 10.3]{aedsh} we have
  \[
  \dwt_f(S_i)=\sum_{\tau\in S_i}2^{-f(\tau)}\le\sum_{\tau\in
    S_i}2^{-\KP(\tau)-i}=2^{-i}\sum_{\tau\in S_i}2^{-\KP(\tau)}<2^{-i}
  \]
  so $S_i$ is a test for $f$-randomness.  Since $X$ is $f$-random it
  follows that $X\notin\bigcap_i\llb S_i\rrb$, i.e., $\exists
  i\,\forall n\,(\KP(X\res n)\ge f(X\res n)-i)$, i.e., $X$ is
  $f$-complex.

  Now suppose $X$ is not $f$-random, say $X\in\bigcap_i\llb A_i\rrb$
  where $A_i$ is uniformly r.e.\ and $\dwt_f(A_i)\le2^{-i}$.  Then
  \[
  \sum_i\sum_{\tau\in
    A_{2i}}2^{-f(\tau)+i}=\sum_i2^i\dwt_f(A_{2i})\le\sum_i2^i2^{-2i}
  =\sum_i2^{-i}=1
  \]
  so by the Kraft/Chaitin Lemma (see \cite[Corollary 10.6]{aedsh}) we
  have
  \begin{center}
    $\exists c\,\forall i\,\forall\tau\,(\tau\in
    A_{2i}\limp\KP(\tau)\le f(\tau)-i+c)$.
  \end{center}
  Since $X\in\bigcap_i\llb A_{2i}\rrb$ it follows that $\exists
  c\,\forall i\,\exists n\,(\KP(X\res n)\le f(X\res n)-i+c)$.  In
  other words, $X$ is not $f$-complex.  This completes the proof.
\end{proof}

\begin{cor}
  \label{cor:rc}
  The sets $S_i=\{\tau\mid\KP(\tau)<f(\tau)-i\}$ form a universal
  test for $f$-randomness.
\end{cor}

\begin{proof}
  Paraphrasing Theorem \ref{thm:rc} we see that $X$ is $f$-random if
  and only if $X\notin\bigcap_i\llb S_i\rrb$.  It remains to prove
  that $\dwt_f(S_i)\le2^{-i}$, but we have already seen this as part
  of the proof of Theorem \ref{thm:rc}.
\end{proof}

The next theorem is a straightforward generalization of
Calude/Staiger/Terwijn \cite[Corollary 4.10]{ca-st-te}.

\begin{thm}
  \label{thm:srsc}
  Let $f:\{0,1\}^*\to[-\infty,\infty]$ be computable.  Then strong
  $f$-randomness is equivalent to strong $f$-complexity.
\end{thm}

\begin{proof}
  Recall that $\KA(\tau)=-\log_2m(\tau)$ where $m:\{0,1\}^*\to[0,1]$
  is a universal left-r.e.\ semimeasure.  See for instance
  \cite[\S3.16]{do-hi-book} or \cite{us-sh}.

  Suppose $X$ is strongly $f$-random.  Let
  $S_i=\{\tau\mid\KA(\tau)<f(\tau)-i\}$.  Clearly $S_i$ is uniformly
  r.e.  We claim that $\pwt_f(S_i)\le2^{-i}$.  To see this, let
  $P\subseteq S_i$ be prefix-free.  Then
  \[
  \dwt_f(P)=\sum_{\tau\in P}2^{-f(\tau)}\le\sum_{\tau\in
    P}2^{-i-\KA(\tau)}=2^{-i}\sum_{\tau\in P}m(\tau)\le2^{-i}
  \]
  since $m$ is a semimeasure.  This proves our claim.  Thus $S_i$ is a
  test for strong $f$-randomness.  Since $X$ is strongly $f$-random,
  we have $X\notin\bigcap_i\llb S_i\rrb$, i.e., $\exists i\,\forall
  n\,(\KA(X\res n)\ge f(X\res n)-i)$, i.e., $X$ is strongly
  $f$-complex.

  Now suppose $X$ is not strongly $f$-random, say $X\in\bigcap_i\llb
  A_i\rrb$ where $A_i$ is uniformly r.e.\ and $\pwt_f(A_i)\le2^{-i}$.
  For each $i$ let $m_i$ be the uniformly left-r.e.\ semimeasure given
  by $m_i(\sigma)=\pwt_f(\{\tau\in A_i\mid\tau\supseteq\sigma\})$.
  Note that $m_i(\tau)\ge\wt_f(\tau)$ whenever $\tau\in A_i$.  For
  each $i$ we have $m_i(\langle\rangle)=\pwt_f(A_i)\le2^{-i}$, hence
  $2^im_{2i}(\langle\rangle)\le2^i2^{-2i}=2^{-i}$, so consider the
  left-r.e.\ semimeasure
  $\overline{m}(\sigma)=\sum_i2^im_{2i}(\sigma)$.  Since $m$ is a
  universal left-r.e.\ semimeasure, let $c$ be such that
  $\overline{m}(\sigma)\le2^cm(\sigma)$ for all $\sigma$.  Then for
  all $\tau\in A_{2i}$ we have
  $2^{i-f(\tau)}=2^i\wt_f(\tau)\le2^im_{2i}(\tau)\le\overline{m}(\tau)
  \le2^cm(\tau)=2^{c-\KA(\tau)}$, hence $\KA(\tau)\le f(\tau)-i+c$.
  Since $X\in\bigcap_i\llb A_{2i}\rrb$ it follows that $\forall
  i\,\exists n\,(\KA(X\res n)\le f(X\res n)-i+c)$.  In other words,
  $X$ is not strongly $f$-complex.  This completes the proof.
\end{proof}

\begin{cor}
  \label{cor:srsc}
  The sets $S_i=\{\tau\mid\KA(\tau)<f(\tau)-i\}$ form a universal
  test for strong $f$-randomness.
\end{cor}

\begin{proof}
  Paraphrasing Theorem \ref{thm:srsc} we see that $X$ is strongly
  $f$-random if and only if $X\notin\bigcap_i\llb S_i\rrb$.  It
  remains to prove that $\pwt_f(S_i)\le2^{-i}$, but we have already
  seen this as part of the proof of Theorem \ref{thm:srsc}.
\end{proof}

\begin{rem}
  \label{rem:s-rand}
  As a special case, consider the functions
  $f_s:\{0,1\}^*\to[0,\infty)$ given by $f_s(\sigma)=s|\sigma|$ where
  $s$ is rational and $0<s\le1$.  Here we are writing $|\sigma|=$ the
  length of $\sigma$.  Define $X\in\{0,1\}^\NN$ to be
  \underline{$s$-random} if it is $f_s$-random, and
  \underline{strongly $s$-random} if it is strongly $f_s$-random.
  Note that Martin-L\"of randomness is equivalent to $1$-randomness
  and to strong $1$-randomness.  The \underline{effective Hausdorff
    dimension} of $X$ is
  \begin{center}
    $\effdim(X)=\sup\{s\mid X$ is $s$-random$\}=\sup\{s\mid X$ is
    strongly $s$-random$\}$
  \end{center}
  and this notion has been studied in
  \cite{miller-dim,reimann-phd,tadaki-A} and many other publications.
\end{rem}

\begin{rem}
  \label{rem:s02}
  Given a computable function $f:\{0,1\}^*\to[-\infty,\infty]$, it is
  easy to see that $\{X\mid X$ is $f$-random$\}$ and $\{X\mid X$ is
  strongly $f$-random$\}$ are $\Sigma^0_2$ subsets of $\{0,1\}^\NN$.
  Conversely, given a $\Sigma^0_2$ set $S\subseteq\{0,1\}^\NN$, we can
  easily construct a computable function $f:\{0,1\}^*\to\NN$ such that
  \begin{center}
    $S=\{X\mid X$ is $f$-random$\}=\{X\mid X$ is strongly
    $f$-random$\}$.
  \end{center}
  Namely, if $S=\bigcup_i\{$paths through $T_i\}$ where
  $T_i\subseteq\{0,1\}^*$, $i\in\NN$ is a computable sequence of
  computable trees, let
  \[
  f(\tau)=\left\{
    \begin{array}{ll}
      1 &\hbox{if }h(\tau\res(|\tau|-1))=h(\tau),\\[4pt]
      2|\tau| &\hbox{otherwise},
    \end{array}
  \right.
  \]
  where $h(\tau)=$ the least $i$ such that $i=|\tau|$ or $\tau\in
  T_i$.  We mention these examples in order to suggest how our
  concepts of $f$-randomness and strong $f$-randomness may apply to a
  wide variety of situations.  See also Theorem \ref{thm:autofr}
  below.
\end{rem}

\section{$g$-randomness implies strong $f$-randomness}
\label{sec:grsfr}

Suppose we have two computable functions
$f,g:\{0,1\}^*\to[-\infty,\infty]$.  Clearly $g$-randomness implies
$f$-randomness provided $\forall\sigma\,(f(\sigma)\le g(\sigma))$.  We
now prove that $g$-randomness implies strong $f$-randomness provided
$g$ grows significantly faster than $f$.  Our result here is a slight
refinement of known results due to Calude/Staiger/Terwijn
\cite{ca-st-te} and Reimann/Stephan \nocite{9asian}\cite{re-st-tests}.
See also Uspensky/Shen \cite[\S4.2]{us-sh}.

\begin{dfn}
  The \underline{increasing set} of $f:\{0,1\}^*\to[-\infty,\infty]$
  is
  \begin{center}
    $I(f)=\{\sigma\mid(\forall\rho\subset\sigma)\,(f(\rho)<f(\sigma))\}$.
  \end{center}
\end{dfn}

\begin{lem}
  \label{lem:int}
  Given a computable function $f:\{0,1\}^*\to[-\infty,\infty]$, we can
  effectively find a computable function
  $\overline{f}:\{0,1\}^*\to\NN$ such that for all $\sigma$,
  \begin{equation}
    \label{eq:fbar}
    f_0(\sigma)<\overline{f}(\sigma)<f_0(\sigma)+2
  \end{equation}
  where $f_0(\sigma)=\min(\max(f(\sigma),0),2|\sigma|)$.  It then
  follows that $f$-randomness is equivalent to
  $\overline{f}$-randomness, and strong $f$-randomness is equivalent
  to strong $\overline{f}$-randomness.
\end{lem}

\begin{proof}
  Given $\sigma\in\{0,1\}^*$ we can effectively approximate
  $f_0(\sigma)$ to find $\overline{f}(\sigma)\in\NN$ such that
  (\ref{eq:fbar}) holds.  In this way we obtain a computable function
  $\overline{f}:\{0,1\}^*\to\NN$.  Using the fact that $\exists
  c\,\forall\sigma\,(0<\KP(\sigma)<2|\sigma|+c$ and
  $0<\KA(\sigma)<2|\sigma|+c)$, we can easily see that (strong)
  $f$-complexity is equivalent to (strong) $\overline{f}$-complexity.
  The desired conclusions then follow in view of Theorems \ref{thm:rc}
  and \ref{thm:srsc}.
\end{proof}

\begin{lem}
  \label{lem:If}
  Let $f:\{0,1\}^*\to\NN$ be computable.  Given an r.e.\ set
  $A\subseteq\{0,1\}^*$ we can effectively find an r.e.\ set
  $\overline{A}\subseteq I(f)$ such that $\llb
  A\rrb\subseteq\llb\overline{A}\rrb$ and
  $\dwt_f(\overline{A})\le\dwt_f(A)$ and
  $\pwt_f(\overline{A})\le\pwt_f(A)$.
\end{lem}

\begin{proof}
  Let $\overline{A}=\{\overline{\sigma}\mid\sigma\in A\}$ where
  $\overline{\sigma}=\min\{\rho\subseteq\sigma\mid f(\rho)\ge
  f(\sigma)\}$.  It is straightforward to verify that this
  $\overline{A}$ has the desired properties.
\end{proof}

\begin{rem}
  Because of Lemmas \ref{lem:int} and \ref{lem:If}, we are often safe
  in assuming that $f:\{0,1\}^*\to\NN$ and that $A\subseteq I(f)$.
\end{rem}

\begin{thm}
  \label{thm:fg}
  Let $f,g:\{0,1\}^*\to[-\infty,\infty]$ be computable with $g$ of the
  form $g(\sigma)=f(\sigma)+h(f(\sigma))$ where $h$ is nondecreasing
  and $\sum_{n=1}^\infty2^{-h(n)}<\infty$.  If $X$ is $g$-random, then
  $X$ is strongly $f$-random.
\end{thm}
 
\begin{proof}
  Because $h$ is nondecreasing, we may safely apply Lemma
  \ref{lem:int} to assume that $f:\{0,1\}^*\to\NN$.  Fix $c$ such that
  $\sum_n2^{-h(n)}\le2^c<\infty$.  Suppose $X$ is not strongly
  $f$-random, say $X\in\bigcap_i\llb A_i\rrb$ where $A_i$ is uniformly
  r.e.\ and $\pwt_f(A_i)\le2^{-i}$.  By Lemma \ref{lem:If} we may
  safely assume that $A_i\subseteq I(f)$ for all $i$.  Let
  $P_{in}=\{\sigma\in A_i\mid f(\sigma)=n\}$.  Clearly
  $A_i=\bigcup_nP_{in}$ and $P_{in}$ is prefix-free.  Thus
  $\dwt_f(P_{in})\le\pwt_f(A_i)$ and
  \[
  \begin{array}{rcl}
    \dwt_g(A_i) &=& \sum_{\sigma\in A_i}2^{-g(\sigma)}\\[6pt]
    &=& \sum_{\sigma\in A_i}2^{-h(f(\sigma))}2^{-f(\sigma)}\\[6pt]
    &=& \sum_n\sum_{\sigma\in P_{in}}2^{-h(n)}2^{-f(\sigma)}\\[6pt]
    &=& \sum_n2^{-h(n)}\sum_{\sigma\in P_{in}}2^{-f(\sigma)}\\[6pt]
    &=& \sum_n2^{-h(n)}\dwt_f(P_{in})\\[6pt]
    &\le& 2^c\pwt_f(A_i)\\[6pt]
    &\le& 2^{c-i}.
  \end{array}
  \]
  Since $X\in\bigcap_i\llb A_i\rrb$ it follows that $X$ is not
  $g$-random, Q.E.D.
\end{proof}

\begin{thm}
  \label{thm:log2}
  Let $f:\{0,1\}^*\to(0,\infty]$ be computable.  Suppose $X$ is\\
  $(f+(1+\epsilon)\log_2f)$-random for some $\epsilon>0$.  Then $X$ is
  strongly $f$-random.
\end{thm}

\begin{proof}
  We may safely assume that $\epsilon$ is rational.  In this case it
  suffices to apply Theorem \ref{thm:fg} with
  $h(x)=(1+\epsilon)\log_2x$.
\end{proof}

\begin{rem}
  \label{rem:hudelson-complex}
  Consider the computable function $f=f_s$ where $s=1/2$, i.e.,
  $f(\sigma)=|\sigma|/2$ for all $\sigma$.  (More generally, let $f$
  be computable and satisfy certain other conditions which we shall
  not specify here.)  Reimann/Stephan \cite{re-st-tests} have
  constructed an $X$ which is $f$-random but not strongly $f$-random.
  Hudelson \cite{hudelson-complex} has constructed an $X$ which is strongly
  $f$-random but such that no $Y$ Turing reducible to $X$ is
  $(f+(1+\epsilon)\log_2f)$-random for any $\epsilon>0$.  We
  conjecture that there exists an $X$ which is $f$-random but such
  that no $Y$ Turing reducible to $X$ is strongly $f$-random.
\end{rem}

\begin{rem}
  \label{rem:log2}
  In Theorem \ref{thm:log2} and Remark \ref{rem:hudelson-complex} we
  may replace $f+(1+\epsilon)\log_2f$ by
  $f+\log_2f+(1+\epsilon)\log_2\log_2f$, etc., as in
  \cite[\S4.2]{us-sh}.
\end{rem}

\section{Propagation of strong $f$-randomness}
\label{sec:sfr}

The purpose of this section is to prove generalizations of Theorems
\ref{thm:mlrXYZ}--\ref{thm:mlrPA} in which Martin-L\"of randomness is
replaced by strong $f$-randomness.  These generalizations are perhaps
the most important new results of this paper.  Let $\mu$ be the
\underline{fair-coin probability measure} on $\{0,1\}^\NN$ given by
$\mu(\llb\sigma\rrb)=2^{-|\sigma|}$.

\begin{dfn}
  \label{dfn:levin}
  A \underline{Levin system} is an indexed family of sets
  $V_\sigma\subseteq\{0,1\}^\NN$, $\sigma\in\{0,1\}^*$, such that
  \begin{enumerate}
  \item $V_\sigma$ is $\Sigma^0_1$ uniformly in $\sigma$,
  \item $V_\sigma\supseteq V_{\sigma\cat\langle0\rangle}\cup
    V_{\sigma\cat\langle1\rangle}$ for all $\sigma$,
  \item $V_{\sigma\cat\langle0\rangle}\cap
    V_{\sigma\cat\langle1\rangle}=\emptyset$ for all $\sigma$.
  \end{enumerate}
  These properties easily imply
  \begin{enumerate}\setcounter{enumi}{3}
  \item $V_\rho\supseteq V_\sigma$ whenever $\rho\subseteq\sigma$,
  \item $V_\sigma\cap V_\tau=\emptyset$ whenever $\sigma$ and $\tau$
    are incompatible.
  \end{enumerate}
\end{dfn}

\begin{lem}
  \label{lem:sfr}
  Let $V_\sigma$ be a Levin system, and let $f$ be computable.  If $X$
  is strongly $f$-random, then $\exists c\,\forall n\,(\mu(V_{X\res
    n})\le2^{c-f(X\res n)})$.
\end{lem}

\begin{proof}
  Let $A_i=\{\sigma\mid\mu(V_\sigma)>2^{i-f(\sigma)}\}$.  Clearly
  $A_i$ is uniformly r.e.  We claim that $\pwt_f(A_i)\le2^{-i}$.  To
  see this, let $P\subseteq A_i$ be prefix-free.  By part 5 of
  Definition \ref{dfn:levin} we have $1\ge\mu(\bigcup_{\sigma\in
    P}V_\sigma)=\sum_{\sigma\in P}\mu(V_\sigma)\ge\sum_{\sigma\in
    P}2^{i-f(\sigma)}=2^i\dwt_f(P)$, so $\dwt_f(P)\le2^{-i}$.  This
  proves our claim.  Thus $A_i$ is a test for strong $f$-randomness.
  Since $X$ is strongly $f$-random, it follows that $X\notin\llb
  A_i\rrb$ for some $i$.  In other words, $\mu(V_{X\res
    n})\le2^{i-f(X\res n)}$ for all $n$, Q.E.D.
\end{proof}

\begin{rem}
  Our idea of using strong $f$-randomness in Lemma \ref{lem:sfr} was
  inspired by Reimann's use of strong $f$-randomness in \cite[Theorem
  14]{reimann-eff}.
\end{rem}

\begin{lem}
  \label{lem:tilde}
  Let $r_\sigma$, $\sigma\in\{0,1\}^*$, be a uniformly left-r.e.\
  system of real numbers.  Given a Levin system $V_\sigma$, we can
  effectively find a Levin system $\widetilde{V}_\sigma$ such that
  \begin{enumerate}
  \item $\widetilde{V}_\sigma\subseteq V_\sigma$ for all $\sigma$,
  \item $\mu(\widetilde{V}_\sigma)\le r_\sigma$ for all $\sigma$,
  \item $\widetilde{V}_\sigma=V_\sigma$ whenever $\sigma$ is such that
    $\mu(V_\rho)<r_\rho$ for all $\rho\subseteq\sigma$.
  \end{enumerate}
\end{lem}

\begin{proof}
  The proof is awkward but straightforward.
\end{proof}

\begin{thm}
  \label{thm:sfrXYZ}
  Let $f:\{0,1\}^*\to[-\infty,\infty]$ be computable.  Suppose $X$ is
  strongly $f$-random and Turing reducible to $Y$ where $Y$ is
  Martin-L\"of random relative to $Z$.  Then $X$ is strongly
  $f$-random relative to $Z$.
\end{thm}

\begin{proof}
  Let $\Phi$ be a partial recursive functional such that $X=\Phi^Y$.
  Consider the Levin system
  $V_\sigma=\{\Ybar\mid\Phi^\Ybar\supseteq\sigma\}$.  By Lemma
  \ref{lem:sfr} let $c$ be such that $\mu(V_{X\res n})<2^{c-f(X\res
    n)}$ for all $n$.  Applying Lemma \ref{lem:tilde} with
  $r_\sigma=2^{c-f(\sigma)}$ we obtain a Levin system
  $\widetilde{V}_\sigma$ such that
  $\mu(\widetilde{V}_\sigma)\le2^{c-f(\sigma)}$ for all $\sigma$, and
  $Y\in V_{X\res n}=\widetilde{V}_{X\res n}$ for all $n$.  Now suppose
  $X$ is not strongly $f$-random relative to $Z$, say
  $X\in\bigcap_i\llb A_i^Z\rrb$ where $A_i^Z$ is uniformly $Z$-r.e.\
  and $\pwt_f(A_i^Z)\le2^{-i}$.  Let $W_i^Z =\bigcup_{\sigma\in
    A_i^Z}\widetilde{V}_\sigma$.  Clearly $W_i^Z$ is uniformly
  $\Sigma^{0,Z}_1$.  Because $X\in\bigcap_i\llb A_i^Z\rrb$ and
  $Y\in\bigcap_nV_{X\res n}=\bigcap_n\widetilde{V}_{X\res n}$, we have
  $Y\in\bigcap_iW_i^Z$.  Let $P_i=\widehat{A}_i^Z=\{$minimal elements
  of $A_i^Z\}$.  Because $\widetilde{V}_\sigma$ is a Levin system, we
  have $W_i^Z=\bigcup_{\sigma\in P_i}\widetilde{V}_\sigma$ and hence
  \[
  \mu(W_i^Z)=\sum_{\sigma\in
    P_i}\mu(\widetilde{V}_\sigma)\le\sum_{\sigma\in
    P_i}2^{c-f(\sigma)}=2^c\dwt_f(P_i)\le2^c\pwt_f(A_i^Z)\le2^{c-i}
  \]
  since $P_i$ is a prefix-free subset of $A_i^Z$.  Thus $Y$ is not
  Martin-L\"of random relative to $Z$, Q.E.D.
\end{proof}

\begin{rem}
  In Theorem \ref{thm:sfrXYZ} the assumption ``$Y$ is Martin-L\"of
  random relative to $Z$'' cannot be weakened to ``$Y$ is strongly
  $f$-random relative to $Z$.''  For example, define $Z(n)=Y(2n)$
  where $Y$ is Martin-L\"of random.  Then $Z$ is strongly $1/2$-random
  (indeed Martin-L\"of random) and Turing reducible to $Y$, and $Y$ is
  strongly $1/2$-random relative to $Z$, but of course $Z$ is not
  strongly $1/2$-random relative to $Z$.
\end{rem}

\begin{thm}
  \label{thm:sfrPA}
  For each $i\in\NN$ let $f_i:\{0,1\}^*\to[-\infty,\infty]$ be
  computable and let $X_i\in\{0,1\}^\NN$.  Suppose $\forall i\,(X_i$
  is strongly $f_i$-random$)$.  Then, we can find $Z$ of $\PA$-degree
  such that $\forall i\,(X_i$ is strongly $f_i$-random relative to
  $Z)$.
\end{thm}

\begin{proof}
  By the Ku\v{c}era/G\'acs Theorem (see \cite[Theorem
  8.3.2]{do-hi-book} or \cite[\S3.3]{nies-book} or \cite[Theorem
  3.8]{aedsh}), let $Y$ be Martin-L\"of random such that $\forall
  i\,(X_i$ is Turing reducible to $Y)$.  By Theorem \ref{thm:mlrPA}
  let $Z$ be of $\PA$-degree such that $Y$ is Martin-L\"of random
  relative to $Z$.  If $\forall i\,(X_i$ is strongly $f_i$-random$)$,
  it follows by Theorem \ref{thm:sfrXYZ} that $\forall i\,(X_i$ is
  strongly $f_i$-random relative to $Z)$.
\end{proof}

\begin{cor}
  \label{cor:sfrPA}
  Let $f:\{0,1\}^*\to[-\infty,\infty]$ be computable.  If $X$ is
  strongly $f$-random, then $X$ is strongly $f$-random relative to
  some $\PA$-degree.
\end{cor}

\begin{proof}
  Apply Theorem \ref{thm:sfrPA} with $X_i=X$ and $f_i=f$ for all $i$.
\end{proof}

Even the following corollary appears to be new.

\begin{cor}
  \label{cor:mlrPA}
  Suppose $(\forall i\in\NN)\,(X_i$ is Martin-L\"of random$)$.  Then,
  we can find $Z$ of $\PA$-degree such that $(\forall i\in\NN)\,(X_i$
  is Martin-L\"of random relative to $Z)$.
\end{cor}

\begin{proof}
  Consider $f:\{0,1\}^*\to[0,\infty)$ where $f(\sigma)=|\sigma|$ for
  all $\sigma$.  By Remark \ref{rem:s-rand} $X_i$ is Martin-L\"of
  random if and only if $X_i$ is strongly $f$-random, and similarly
  $X_i$ is Martin-L\"of random relative to $Z$ if and only if $X_i$ is
  strongly $f$-random relative to $Z$.  Apply Theorem \ref{thm:sfrPA}
  with $f_i=f$ for all $i$.
\end{proof}

We end this section by presenting a kind of Borel/Cantelli Lemma for
strong $f$-randomness.  Let us say that $X$ is \underline{strongly
  $\BC$-$f$-random} if $\{i\mid X\in\llb A_i\rrb\}$ is finite whenever
$A_i$ is uniformly r.e.\ and $\sum_i\pwt_f(A_i)<\infty$.  This notion
resembles a generalization of Tadaki's earlier notion of Solovay
$D$-randomness \cite[Definition 3.8]{tadaki-A}.

\begin{thm}
  \label{thm:sBCfrXYZ}
  Let $f:\{0,1\}^*\to[-\infty,\infty]$ be computable.  Suppose $X$ is
  strongly $f$-random and Turing reducible to $Y$ where $Y$ is
  Martin-L\"of random relative to $Z$.  Then $X$ is strongly
  $\BC$-$f$-random relative to $Z$.
\end{thm}

\begin{proof}
  Suppose $X$ is not strongly $\BC$-$f$-random relative to $Z$.  Let
  $A_i^Z$ be uniformly $Z$-r.e.\ such that
  $\sum_i\pwt_f(A_i^Z)<\infty$ and $X\in\llb A_i^Z\rrb$ for infinitely
  many $i$.  Let $V_\sigma$, $c$, $\widetilde{V}_\sigma$, $W_i^Z$,
  $P_i$ be as in the proof of Theorem \ref{thm:sfrXYZ}.  For all $i$
  we have $\mu(W_i^Z)\le2^c\pwt_f(A_i^Z)$, hence
  $\sum_i\mu(W_i^Z)\le2^c\sum_i\pwt_f(A_i^Z)<\infty$.  On the other
  hand, for all $i$ such that $X\in\llb A_i^Z\rrb$ we have $Y\in
  W_i^Z$, so $Y\in W_i^Z$ for infinitely many $i$.  Relativizing
  Solovay's Lemma \cite[Lemma 3.5]{aedsh} to $Z$, we see that $Y$ is
  not Martin-L\"of random relative to $Z$, Q.E.D.
\end{proof}

\begin{thm}
  \label{thm:sBCfrPA}
  Let $f:\{0,1\}^*\to[-\infty,\infty]$ be computable.  If $X$ is
  strongly $f$-random, then $X$ is strongly $\BC$-$f$-random relative
  to some $\PA$-degree.
\end{thm}

\begin{proof}
  By the Ku\v{c}era/G\'acs Theorem, let $Y$ be Martin-L\"of random
  such that $X$ is Turing reducible to $Y$.  By Theorem
  \ref{thm:mlrPA} let $Z$ be of $\PA$-degree such that $Y$ is
  Martin-L\"of random relative to $Z$.  If $X$ is strongly $f$-random,
  Theorem \ref{thm:sBCfrXYZ} tells us that $X$ is strongly
  $\BC$-$f$-random relative to $Z$, Q.E.D.
\end{proof}

\begin{cor}
  \label{cor:sBCfr}
  Let $f:\{0,1\}^*\to[-\infty,\infty]$ be computable.  Then strong
  $f$-randomness is equivalent to strong $\BC$-$f$-randomness.
\end{cor}

\begin{proof}
  Trivially strong $\BC$-$f$-randomness implies strong $f$-randomness.
  The converse is immediate from Theorem \ref{thm:sBCfrPA}.
\end{proof}

\begin{rem}
  \label{rem:sBCfr}
  It is possible to give a direct proof of Corollary \ref{cor:sBCfr}
  resembling the standard proof of Solovay's Lemma \cite[Lemma
  3.5]{aedsh}.
\end{rem}

\section{Propagation of non-K-triviality}
\label{sec:nkt}

Recall from \cite{do-hi-book,nies-book} that $X$ is
\underline{$\LR$-reducible} to $Z$, abbreviated $X\leLR Z$, if
$\forall Y\,((Y$ Martin-L\"of random relative to $Z)\limp(Y$
Martin-L\"of random relative to $X))$.  The concept of
$\LR$-recucibility has been very useful \cite{kj-mi-so,aedsh,massmtr}
in the reverse mathematics of measure-theoretic regularity.  It is
also known (see \cite[Chapter 11]{do-hi-book} or \cite[Chapter
5]{nies-book}) that $\LR$-reducibility can be used to characterize
\underline{K-triviality}.  Namely, $X$ is K-trivial if and only if
$X\leLR0$.

From our point of view in this paper, it seems reasonable to view
non-K-triviality as a kind of partial randomness notion.  Accordingly,
we now present appropriate analogs of our main propagation results,
Theorems \ref{thm:sfrXYZ} and \ref{thm:sfrPA}.  Our results in this
section are easy consequences of previously known characterizatons of
K-triviality.

\begin{thm}
  \label{thm:nktXYZ}
  Suppose $X$ is Turing reducible to $Y$ where $Y$ is Martin-L\"of
  random relative to $Z$.  Then $X\nleLR0$ implies $X\nleLR Z$.
\end{thm}

\begin{proof}
  Since $X\nleLR0$, it follows by \cite[Chapter 11]{do-hi-book} or
  \cite[Chapter 5]{nies-book} that $X$ is not a \emph{base for
    Martin-L\"of randomness}.  In particular, since $X$ is Turing
  reducible to $Y$, $Y$ is not Martin-L\"of random relative to $X$.
  But then, since $Y$ is Martin-L\"of random relative to $Z$, we have
  $X\nleLR Z$, Q.E.D.
\end{proof}

\begin{thm}
  \label{thm:nktPA}
  Suppose $X_i\nleLR0$ for all $i\in\NN$.  Then, we can find $Z$ of
  $\PA$-degree such that $X_i\nleLR Z$ for all $i\in\NN$.
\end{thm}

\begin{proof}
  For each $i$ let $Y_i$ be Martin-L\"of random but not Martin-L\"of
  random relative to $X_i$.  By Corollary \ref{cor:mlrPA} let $Z$ be
  of $\PA$-degree such that $\forall i\,(Y_i$ is Martin-L\"of random
  relative to $Z)$.  It follows that $\forall i\,(X_i\nleLR Z)$,
  Q.E.D.
\end{proof}

\begin{rem}
  In Theorems \ref{thm:nktXYZ} and \ref{thm:nktPA}, the conclusion
  $X\nleLR Z$ implies that $X\oplus Z\nleLR Z$, i.e., $X$ is not
  K-trivial relative to $Z$.  On the other hand, results such as
  Theorems \ref{thm:mlrPA} and \ref{thm:sfrPA} and \ref{thm:nktPA}
  bear an obvious resemblance to the well known \emph{GKT Theorem}
  (see Gandy/Kreisel/Tait \cite{gkt} or \cite[Theorem 2.5]{js} or
  \cite[Theorem VIII.2.24]{sosoa}).  Indeed, Theorem \ref{thm:nktPA}
  is just the GKT Theorem with Turing reducibility replaced by
  $\LR$-reducibility.
\end{rem}

\section{Propagation of diagonal nonrecursiveness}
\label{sec:dnr}

Let $\{n\}$ denote the partial recursive functional with index $n$.
Let $\DNR$ be the set of functions $f:\NN\to\NN$ which are
\underline{diagonally nonrecursive}, i.e., $f(n)\ne\{n\}(n)$ for all
$n$.  Known results concerning diagonal nonrecursiveness may be found
in \nocite{lmps8}\cite{askhls,jockusch-fpf,kms-tams,massrand}.  We
also consider \underline{relative $\DNR$-ness}:
$\DNR^Z=\{f\in\NN^\NN\mid\forall n\,(f(n)\ne\{n\}^Z(n))\}$.  The
purpose of this section is to obtain propagation results for diagonal
nonrecursiveness.

\begin{thm}
  \label{thm:dnrXYZ}
  Suppose there exists a $\DNR$ function which is Turing reducible to
  $X$.  Suppose also that $X$ is Turing reducible to $Y$ where $Y$ is
  Martin-L\"of random relative to $Z$.  Then there exists a $\DNR^Z$
  function which is Turing reducible to $X$.
\end{thm}

In order to prove Theorem \ref{thm:dnrXYZ} we need the following
lemma, which is a variant of the Parametrized Recursion Theorem.  In
stating and proving our lemma, we shall use standard
recursion-theoretic notation.  In particular, for any expression $E$
we write $E{\downarrow}$ to mean that $E$ is defined, and
$E{\uparrow}$ to mean that $E$ is undefined.  We also write $E_1\simeq
E_2$ to mean that either ($E_1{\downarrow}$ and $E_2{\downarrow}$ and
$E_1=E_2)$ or $(E_1{\uparrow}$ and $E_2{\uparrow})$.  Via G\"odel
numbering, we identify finite sequences of positive integers with
positive integers.  We write $f\leT X$ to mean that $f$ is Turing
reducible to $X$.

\begin{lem}
  \label{lem:pnj}
  Let $\Theta(n,j,\sigma,-)$ be a partial recursive functional.
  Then, we can find a primitive recursive function $p(n,j)$ such that
  \[
  \{p(n,j)\}(-)\,\,\simeq\,\,\Theta(n,j,\langle p(n,i)\mid i\le
  j\rangle,-)
  \]
  for all $n,j,-$.
\end{lem}

\begin{proof}
  By the Parametrized Recursion Theorem, let $q$ be a primitive
  recursive function such that
  $\{q(n,j,\sigma)\}(-)\simeq\Theta(n,j,\sigma\cat\langle
  q(n,j,\sigma)\rangle,-)$ for all $n,j,\sigma,-$.  Define $p$
  primitive recursively by letting $p(n,j)=q(n,i,\langle p(n,i)\mid
  i<j\rangle)$ for all $n,j$.  We then have
  \[
  \begin{array}{rcl}
    \{p(n,j)\}(-) &\simeq& \{q(n,j,\langle p(n,i)\mid
    i<j\rangle)\}(-)\\[4pt]
    &\simeq& \Theta(n,j,\langle p(n,i)\mid i<j\rangle\cat\langle
    q(n,j,\langle p(n,i)\mid i<j\rangle)\rangle,-)\\[4pt]
    &\simeq& \Theta(n,j,\langle p(n,i)\mid i<j\rangle\cat\langle
    p(n,j)\rangle,-)\\[4pt]
    &\simeq& \Theta(n,j,\langle p(n,i)\mid i\le j\rangle,-)
  \end{array}
  \]
  and this proves our lemma.
\end{proof}

We now prove Theorem \ref{thm:dnrXYZ}.

\begin{proof}[Proof of Theorem \ref{thm:dnrXYZ}]
  Let $f\in\NN^\NN$ be $\DNR$ and $\leT X$.  Then $f\leT Y$ so let
  $\Phi$ be a partial recursive functional such that $f=\Phi^Y$, i.e.,
  $f(n)=\Phi(Y,n)$ for all $n$.  As in \S\ref{sec:sfr} let $\mu$ be
  the fair-coin probability measure on $\{0,1\}^\NN$.  Define a
  partial recursive function $\theta(n,j,\sigma)\simeq$ some $m$ such
  that $\mu(\{\Ybar\mid\Phi^\Ybar(\sigma(j)){\downarrow}=m$ and
  $(\forall i<j)\,(\Phi^\Ybar(\sigma(i)){\downarrow}\ne
  \{\sigma(i)\}(\sigma(i)){\downarrow})\})>2^{-n}$.  Apply Lemma
  \ref{lem:pnj} to obtain a primitive recursive function $p(n,j)$ such
  that $\{p(n,j)\}(p(n,j))\simeq$ some $m$ such that
  $\mu(V_{n,j,m})>2^{-n}$ where
  $V_{n,j,m}=\{\Ybar\mid\Phi^\Ybar(p(n,j)){\downarrow}=m$ and
  $(\forall i<j)\,(\Phi^\Ybar(p(n,i)){\downarrow}\ne
  \{p(n,i)\}(p(n,i)){\downarrow})\}$.  Thus
  $\{p(n,j)\}(p(n,j)){\downarrow}$ implies $\mu(V_{n,j})>2^{-n}$ where
  $V_{n,j}=V_{n,j,\{p(n,j)\}(p(n,j))}$.  On the other hand, $i\ne j$
  implies $V_{n,i}\cap V_{n,j}=\emptyset$ so for each $n$ there is at
  least one $j\le2^n$ such that $\{p(n,j)\}(p(n,j)){\uparrow}$.
  
  Let $\Psi$ be a partial recursive functional defined by
  \begin{center}
    $\Psi^\Ybar(n)\simeq\langle\Phi^\Ybar(p(n,i))\mid i\le2^n\rangle$
  \end{center}
  for all $n$.  In particular we have $g\in\NN^\NN$ defined by
  \begin{center}
    $g(n)=\Psi^Y(n)=\langle f(p(n,i))\mid i\le2^n\rangle$
  \end{center}
  for all $n$.  Clearly $g\leT f\leT X$, so it will suffice to prove
  that $g(n)\ne\{n\}^Z(n)$ for all but finitely many $n$.

  Let $U_n^Z=\{\Ybar\mid\Psi^\Ybar(n){\downarrow}=\{n\}^Z(n){\downarrow}\}$.
  Clearly $U_n^Z$ is uniformly $\Sigma^{0,Z}_1$.  Given a rational
  number $r$, let $U_n^Z[r]$ be $U_n^Z$ enumerated so long as its
  $\mu$-measure is $\le r$.  Thus $U_n^Z[r]$ is uniformly
  $\Sigma^{0,Z}_1$ and $\mu(U_n^Z[r])\le r$.  Moreover,
  $U_n^Z[r]=U_n^Z$ if and only if $\mu(U_n^Z)\le r$.  Since $Y$ is
  Martin-L\"of random, it follows by Solovay's Lemma \cite[Lemma
  3.5]{aedsh} that $Y\notin U_n^Z[2^{-n}]$ for all but finitely many
  $n$.  Therefore, it will suffice to prove $g(n)\ne\{n\}^Z(n)$ for
  all such $n$.

  Supposing otherwise, we would have $\Psi^Y(n)=g(n)=\{n\}^Z(n)$,
  hence $Y\in U_n^Z$, hence $\mu(U^Z_n)>2^{-n}$.  Moreover, for all
  $\Ybar\in U_n^Z$ we would have $\Psi^\Ybar(n)=\{n\}^Z(n)=g(n)$, hence
  $\Phi^\Ybar(p(n,i))=f(p(n,i))\ne\{p(n,i)\}(p(n,i))$ for all $i\le2^n$.
  Let $j\le2^n$ be such that $(\forall
  i<j)\,(\{p(n,i)\}(p(n,i)){\downarrow})$.  Then $U_n^Z\subseteq
  V_{n,j,f(p(n,j))}$, hence
  $\mu(V_{n,j,f(p(n,j))})\ge\mu(U_n^Z)>2^{-n}$, hence
  $\{p(n,j)\}(p(n,j)){\downarrow}$, so by induction on $j$ we see that
  $\{p(n,j)\}(p(n,j)){\downarrow}$ holds for all $j\le2^n$.  This
  contradiction completes the proof.
\end{proof}

\begin{thm}
  \label{thm:dnrQ}
  Let $Q$ be a nonempty $\Pi^0_1$ subset of $\{0,1\}^\NN$.  If
  $(\forall i\in\NN)\,(\exists f\in\DNR)\,(f\leT X_i)$, then $(\exists
  Z\in Q)\,(\forall i\in\NN)\,(\exists g\in\DNR^Z)\,(g\leT X_i)$.
\end{thm}

\begin{proof}
  By the Ku\v{c}era/G\'acs Theorem (see \cite[Theorem
  8.3.2]{do-hi-book} or \cite[\S3.3]{nies-book} or \cite[Theorem
  3.8]{aedsh}), let $Y$ be Martin-L\"of random such that $\forall
  i\,(X_i\leT Y)$.  By Theorem \ref{thm:mlrQ} let $Z\in Q$ be such
  that $Y$ is Martin-L\"of random relative to $Z$.  If $\forall
  i\,\exists f\,(f\in\DNR$ and $f\leT X_i)$, it follows by Theorem
  \ref{thm:dnrXYZ} that $\forall i\,\exists g\,(g\in\DNR^Z$ and $g\leT
  X_i)$.
\end{proof}

\begin{cor}
  \label{cor:dnrQ}
  Let $Q$ be a nonempty $\Pi^0_1$ subset of $\{0,1\}^\NN$.  If there
  exists a $\DNR$ function which is Turing reducible to $X$, then for
  some $Z\in Q$ there exists a $\DNR^Z$ function which is Turing
  reducible to $X$.
\end{cor}

\begin{proof}
  This is the special case of Theorem \ref{thm:dnrQ} with $X_i=X$ for
  all $i\in\NN$.
\end{proof}

\begin{thm}
  \label{thm:dnrPA}
  Suppose $(\forall i\in\NN)\,(\exists f\in\DNR)\,(f\leT X_i)$.  Then,
  there exists $Z$ of $\PA$-degree such that $(\forall
  i\in\NN)\,(\exists g\in\DNR^Z)\,(g\leT X_i)$.
\end{thm}

\begin{proof}
  In Theorem \ref{thm:dnrQ} let $Q$ be the $\Pi^0_1$ set consisting of
  all completions of first-order Peano arithmetic.
\end{proof}

\begin{cor}
  \label{cor:dnrPA}
  If there exists a $\DNR$ function which is Turing reducible to $X$,
  then for some $Z$ of $\PA$-degree there exists a $\DNR^Z$ function
  which is Turing reducible to $X$.
\end{cor}

\begin{proof}
  In Corollary \ref{cor:dnrQ} let $Q$ be the $\Pi^0_1$ set consisting
  of all completions of first-order Peano arithmetic.
\end{proof}

\begin{rem}
  \label{rem:Cbdnr}
  As in \cite[\S10]{massrand} and \cite[\S2.2]{masseff}, let $C$ be a
  ``nice'' class of recursive functions.  For example, $C$ could be
  the class of \emph{all} recursive functions, or the class of
  primitive recursive functions, or the class of recursive functions
  up to level $\alpha$ of the transfinite Ackermann hierarchy for some
  limit ordinal $\alpha\le\varepsilon_0$.  A function $f:\NN\to\NN$ is
  said to be \underline{$C$-bounded} if $(\exists F\in C)\,\forall
  n\,(f(n)<F(n))$.  In particular, $f$ is \underline{recursively
    bounded} if it is $C$-bounded where $C=$ the class of all
  recursive functions.  Our proofs above show that Theorems
  \ref{thm:dnrXYZ} and \ref{thm:dnrQ} and \ref{thm:dnrPA} and
  Corollaries \ref{cor:dnrQ} and \ref{cor:dnrPA} also hold with
  ``$\DNR$'' replaced by ``$C$-bounded $\DNR$.''  It suffices to note
  that, in our proof of Theorem \ref{thm:dnrXYZ}, if $f$ is
  $C$-bounded then so is $g$.  See also the refinements mentioned in
  Remarks \ref{rem:fine1} and \ref{rem:fine2} below.
\end{rem}

We end this section by presenting an alternative proof of Corollary
\ref{cor:dnrQ}.

\begin{proof}[Alternative proof of Corollary \ref{cor:dnrQ}]
  Let $\NN^*$ be the set of finite sequences of positive integers.
  For each $\sigma\in\NN^*$ let
  \begin{center}
    $Q_\sigma=\{Z\in Q\mid(\forall
    n<|\sigma|)\,(\sigma(n)\ne\{n\}^Z(n))\}$
  \end{center}
  where $|\sigma|=$ the length of $\sigma$.  Clearly
  $\sigma\subseteq\tau$ implies $Q_\sigma\supseteq Q_\tau$.  By the
  Parametrization or S-m-n Theorem, let $p(n,\sigma)$ be a primitive
  recursive function such that for all $m$,
  $\{p(n,\sigma)\}(p(n,\sigma))=m$ if and only if $\{n\}^Z(n)=m$ for
  all $Z\in Q_\sigma$.  Let $f\le_TX$ be a $\DNR$ function.  Define
  $g\le_TX$ recursively by letting $g(n)=f(p(n,\langle g(i)\mid
  i<n\rangle))$ for all $n$.  We are going to show that $g$ is $\DNR$
  relative to some $Z\in Q$.

  We claim that $Q_{\langle g(i)\mid i<n\rangle}\ne\emptyset$ for all
  $n$.  To begin with, we have $Q_{\langle\rangle}=Q\ne\emptyset$.
  Assume inductively that $Q_{\langle g(i)\mid
    i<n\rangle}\ne\emptyset$.  We shall prove that $Q_{\langle
    g(i)\mid i\le n\rangle}\ne\emptyset$.  There are two cases.  If
  $\{p(n,\langle g(i)\mid i<n\rangle)\}(p(n,\langle g(i)\mid
  i<n\rangle))=m$, we have $\{n\}^Z(n)=m$ for all $Z\in Q_{\langle
    g(i)\mid i<n\rangle}$, but $g(n)=f(p(n,\langle g(i)\mid
  i<n\rangle))\ne m$ since $f$ is $\DNR$.  Thus $Q_{\langle g(i)\mid
    i\le n\rangle}=Q_{\langle g(i)\mid i<n\rangle}\ne\emptyset$.  If
  $\{p(n,\langle g(i)\mid i<n\rangle)\}(p(n,\langle g(i)\mid
  i<n\rangle))$ is undefined, there exists $Z\in Q_{\langle g(i)\mid
    i<n\rangle}$ such that $\{n\}^Z(n)\ne g(n)$, and then $Z$ belongs
  to $Q_{\langle g(i)\mid i\le n\rangle}$.  This proves our claim.

  By compactness, our claim implies that $\bigcap_{n=0}^\infty
  Q_{\langle g(i)\mid i<n\rangle}\ne\emptyset$.  Moreover, from the
  definition of $Q_{\langle g(i)\mid i<n\rangle}$ we see that $g$ is
  $\DNR$ relative to any $Z\in \bigcap_{n=0}^\infty Q_{\langle
    g(i)\mid i<n\rangle}$.  This completes the proof.
\end{proof}

\begin{rem}
  Our alternative proof of Corollary \ref{cor:dnrQ} is more
  constructive than the previous proof via Theorem \ref{thm:dnrXYZ}
  and the Ku\v{c}era/G\'acs Theorem.  In particular, the alternative
  proof can be formalized in $\WKLo$ (see \cite{sosoa}) while the
  previous proof cannot.

  There are some issues here which are interesting from the viewpoint
  of reverse mathematics \cite{sosoa}.  For example, consider the
  following statement.
  \begin{quote}
    Let $Q$ be a nonempty $\Pi^0_1$ subset of $\{0,1\}^\NN$.  If $X_1$
    and $X_2$ are Martin-L\"of random, there exists $Z\in Q$ such that
    $X_1$ and $X_2$ are Martin-L\"of random relative to $Z$.
  \end{quote}
  By Corollary \ref{cor:mlrPA} this statement is true, and from the
  truth of the statement it follows easily that the statement is true
  in all $\omega$-models of $\WKLo$.  Moreover, we conjecture that the
  statement is provable in $\WKLo$.  On the other hand, by
  \cite[Theorem 2.1]{ba-le-ng} together with
  \nocite{lc2002}\cite{stephan}, the following special case of the
  Ku\v{c}era/G\'acs Theorem is false in all $\omega$-models of $\WKLo$
  except those which contain $0^{(1)}=$ the Turing jump of $0$.
  \begin{quote}
    If $X_1$ and $X_2$ are Martin-L\"of random, there exists a
    Martin-L\"of random $Y$ such that $X_1\leT Y$ and $X_2\leT Y$.
  \end{quote}
\end{rem}

\section{Propagation of autocomplexity}
\label{sec:auto}

In this section we prove propagation results for \emph{autocomplexity}
and \emph{complexity}.  We also obtain a characterization of
autocomplexity in terms of $f$-randomness and strong $f$-randomness.
With this characterization plus \cite[Theorem 2.3]{kms-tams}, we see
considerable overlap between the propagation results of this section
and those of \S\ref{sec:sfr} and \S\ref{sec:dnr}.

\begin{dfn}
  \label{dfn:auto}
  {\ }
  \begin{enumerate}
  \item Following \cite{kms-tams} we define $X\in\{0,1\}^\NN$ to be
    \underline{autocomplex} if there exists an unbounded function
    $h:\NN\to\NN$ such that $h\leT X$ and $h(n)\le\KS(X\res n)$ for
    all $n$.  Here $\KS$ denotes \underline{simple Kolmogorov
      complexity} \cite{us-sh}, also known as \underline{plain
      complexity} \cite{do-hi-book,nies-book}.
  \item Following \cite{binns-complex} and \cite{hi-ki-null} and
    \cite{kms-tams}, we define $X\in\{0,1\}^\NN$ to be
    \underline{complex} if there exists an unbounded computable
    function $h:\NN\to\NN$ such that $h(n)\le\KS(X\res n)$ for all
    $n$.
  \end{enumerate}
\end{dfn}

\begin{rem}
  \label{rem:K}
  By \cite[\S4.3.1]{us-sh} there exist constants $c_1$ and $c_2$ such
  that
  $\KS(\sigma)\le\KP(\sigma)+c_1\le\KS(\sigma)+3\log_2|\sigma|+c_2$
  and
  $\KA(\sigma)\le\KP(\sigma)+c_1\le\KA(\sigma)+3\log_2|\sigma|+c_2$
  for all $\sigma\in\{0,1\}^*$.  These inequalities imply that the
  distinctions among $\KS$ and $\KP$ and $\KA$ are immaterial for some
  purposes.  In particular, we can replace $\KS$ in Definition
  \ref{dfn:auto} by $\KP$ or $\KA$.
\end{rem}

We begin with autocomplexity.

\begin{thm}
  \label{thm:autofr}
  The following are pairwise equivalent.
  \begin{enumerate}
  \item $X$ is autocomplex.
  \item $X$ is $f$-random for some computable $f:\{0,1\}^*\to\NN$ such
    that $\{f(X\res n)\mid n\in\NN\}$ is unbounded.
  \item $X$ is strongly $f$-random for some computable
    $f:\{0,1\}^*\to\NN$ such that $\{f(X\res n)\mid n\in\NN\}$  is
    unbounded.
  \end{enumerate}
\end{thm}

\begin{proof}
  The equivalence $2\liff3$ is clear in view of Remark \ref{rem:K}.

  To prove $2\limp1$, suppose $2$ holds via $f$.  By Theorem
  \ref{thm:rc} $X$ is $f$-complex, so let $c\in\NN$ be such that
  $\KP(X\res n)\ge f(X\res n)-c$ for all $n$.  Then for all $n$ we
  have $\KP(X\res n)\ge h(n)$ where $h(n)=\max(1,f(X\res n)-c)$.
  Clearly $h\leT X$ (in fact $h$ is Lipschitz computable from $X$) and
  $h$ is unbounded, so it follows by Remark \ref{rem:K} that $X$ is
  autocomplex, i.e., $1$ holds.

  It remains to prove $1\limp2$.  Suppose $X$ is autocomplex.  By
  Remark \ref{rem:K} let $h:\NN\to\NN$ be unbounded such that $h\leT
  X$ and $h(n)\le\KP(X\res n)$ for all $n$.  Let $\Phi$ be a partial
  recursive functional such that $h=\Phi^X$.  Consider the primitive
  recursive function $f:\{0,1\}^*\to\NN$ defined by
  $f(\sigma)=\max\{p(\sigma,n)\mid n\le|\sigma|\}$ where
  $p(\sigma,n)=\Phi^{\sigma}_{|\sigma|}(n)$ if
  $\Phi^{\sigma}_{|\sigma|}(n){\downarrow}$, otherwise
  $p(\sigma,n)=1$.  Then for all $n$ and all sufficiently large $m\ge
  n$ we have $h(n)=p(X\res m,n)\le f(X\res m)$.  Since $\{h(n)\mid
  n\in\NN\}$ is unbounded, it follows that $\{f(X\res m)\mid
  m\in\NN\}$ is unbounded.  Consider the primitive recursive function
  $q(\sigma)=$ the least $n\le|\sigma|$ such that
  $f(\sigma)=p(\sigma,n)$.  Let $c$ be a constant such that
  $\KP(\sigma\res q(\sigma))\le\KP(\sigma)+c$ for all $\sigma$.  Then
  for all $m$ we have $\KP(X\res m)+c\ge\KP(X\res q(X\res m))\ge
  h(X\res q(X\res m))\ge p(X\res m,q(X\res m))=f(X\res m)$ so $X$ is
  $f$-complex.  It follows by Theorem \ref{thm:rc} that $X$ is
  $f$-random.  This completes the proof.
\end{proof}

\begin{thm}
  \label{thm:pauto}
  {\ }
  \begin{enumerate}
  \item If $X$ is autocomplex and $\leT Y$ where $Y$ is Martin-L\"of
    random relative to $Z$, then $X$ is autocomplex relative to $Z$.
  \item If $(\forall i\in\NN)\,(X_i$ is autocomplex$)$, there exists
    $Z$ of $\PA$-degree such that $(\forall i\in\NN)\,(X_i$ is
    autocomplex relative to $Z)$.
  \end{enumerate}
\end{thm}

\begin{proof}[First proof]
  Part 1 is immediate from Theorems \ref{thm:sfrXYZ} and
  \ref{thm:autofr}.  Part 2 is immediate from Theorems \ref{thm:sfrPA}
  and \ref{thm:autofr}.
\end{proof}

\begin{proof}[Second proof]
  By Kjos-Hanssen/Merkle/Stephan \cite[Theorem 2.3]{kms-tams} we know
  that $X$ is autocomplex if and only if there exists a $\DNR$
  function which is Turing reducible to $X$.  Modulo this result,
  parts 1 and 2 are equivalent to Theorems \ref{thm:dnrXYZ} and
  \ref{thm:dnrPA} respectively.
\end{proof}

\begin{rem}
  \label{rem:autoXYZ}
  Yet another proof of Theorem \ref{thm:pauto} was obtained
  independently by Bienvenu \cite{bienvenu-dnrXYZ} who had seen it
  conjectured in an earlier draft of the present paper.  The earlier
  draft included Theorems \ref{thm:sfrXYZ} and \ref{thm:sfrPA}, as
  well as Corollary \ref{cor:dnrQ} with our alternative proof, but it
  did not include Theorem \ref{thm:dnrXYZ} or \ref{thm:dnrPA} or
  \ref{thm:autofr}.
\end{rem}

We now turn to propagation results for complexity.  Let us define
$f:\{0,1\}^*\to[-\infty,\infty]$ to be \underline{length-invariant} if
$\forall\sigma\,\forall\tau\,(|\sigma|=|\tau|\limp
f(\sigma)=f(\tau))$.

\begin{thm}
  \label{thm:cmpfr}
  The following are pairwise equivalent.
  \begin{enumerate}
  \item $X$ is complex.
  \item $X$ is $f$-random for some computable $f:\{0,1\}^*\to\NN$
    which is unbounded and length-invariant.
  \item $X$ is strongly $f$-random for some computable
    $f:\{0,1\}^*\to\NN$ which is unbounded and length-invariant.
  \end{enumerate}
\end{thm}

\begin{proof}
  This is immediate from Theorems \ref{thm:rc} and \ref{thm:srsc} and
  Remark \ref{rem:K}.
\end{proof}

\begin{thm}
  \label{thm:pcmp}
  {\ }
  \begin{enumerate}
  \item If $X$ is complex and $\leT Y$ where $Y$ is Martin-L\"of
    random relative to $Z$, then $X$ is complex relative to $Z$.
  \item If $(\forall i\in\NN)\,(X_i$ is complex$)$, there exists $Z$
    of $\PA$-degree such that $(\forall i\in\NN)\,(X_i$ is complex
    relative to $Z)$.
  \end{enumerate}
\end{thm}

\begin{proof}
  Part 1 is immediate from Theorems \ref{thm:sfrXYZ} and
  \ref{thm:cmpfr}.  Part 2 is immediate from Theorems \ref{thm:sfrPA}
  and \ref{thm:cmpfr}.
\end{proof}

\begin{rem}
  \label{rem:fine1}
  By \cite[Theorem 2.3]{kms-tams} we know that $X$ is complex if and
  only if some $\DNR$ function is truth-table reducible to $X$.
  Consequently, the Turing degrees of complex $X$'s are the same as
  the Turing degrees of recursively bounded $\DNR$ functions.  And of
  course, the Turing degrees of autocomplex $X$'s are the same as the
  Turing degrees of $\DNR$ functions.  Thus Theorems \ref{thm:sfrXYZ}
  and \ref{thm:sfrPA} may be viewed as far-reaching refinements not
  only of Theorems \ref{thm:pauto} and \ref{thm:pcmp} but also of
  Theorems \ref{thm:dnrXYZ}--\ref{thm:dnrPA} and Remark
  \ref{rem:Cbdnr}.
\end{rem}

\begin{rem}
  \label{rem:fine2}
  By \cite[Theorem 1.8]{askhls} there exists an autocomplex $X$ such
  that no complex $Y$ is Turing reducible to $X$.  Within the class of
  complex $X$'s, much more refined results of the same kind have been
  obtained by Hudelson \cite{hudelson-complex} generalizing the main
  result of Miller \cite[Theorem 4.1]{miller-dim}. See also Remarks
  \ref{rem:hudelson-complex} and \ref{rem:log2} above, as well as
  \cite[\S\S2.1--2.3, Figure 1]{masseff}.
\end{rem}

\section{Vehement $f$-randomness }
\label{sec:vfr}

In this section we define vehement $f$-randomness and discuss its
relationship with strong $f$-randomness.  The notion of vehement
$f$-randomness was originally introduced by Kjos-Hanssen (unpublished,
but see \cite{reimann-eff}).  We prove that, under a convexity
hypothesis on $f$, vehement $f$-randomness is equivalent to strong
$f$-randomness.  Our result is a generalization of known results due
to Reimann \cite[Corollary 21]{reimann-eff} and Miller \cite[Lemma
3.3]{miller-dim}.

\begin{dfn}
  \label{dfn:vwt} 
  Given $f:\{0,1\}^*\to[-\infty,\infty]$, the \underline{vehement
    $f$-weight} of $A\subseteq\{0,1\}^*$ is defined as
  $\vwt_f(A)=\inf\{\dwt_f(S)\mid\llb A\rrb\subseteq\llb S\rrb\}$.
\end{dfn}

\begin{rem}
  \label{rem:vwt}
  Note that $\llb A\rrb\subseteq\llb B\rrb$ implies
  $\vwt_f(A)\le\vwt_f(B)$.  In particular, $\vwt_f(A)$ depends only on
  $\llb A\rrb$.
\end{rem}

\begin{lem}
  \label{lem:vlep}
  For all $A$ we have $\vwt_f(A)\le\dwt_f(\widehat{A})\le\pwt_f(A)$.
\end{lem}

\begin{proof}
  The first inequality holds because $\llb
  A\rrb\subseteq\llb\widehat{A}\rrb$.  The second inequality holds
  because $\widehat{A}$ is a prefix-free subset of $A$.
\end{proof}

\begin{dfn}
  \label{dfn:good}
  Fix $f:\{0,1\}^*\to[-\infty,\infty]$.  A \underline{good cover} of
  $A$ is a set $B$ such that $\llb A\rrb\subseteq\llb B\rrb$ and
  $\pwt_f(B)\le\vwt_f(A)$.  It follows by Remark \ref{rem:vwt} and
  Lemma \ref{lem:vlep} that
  $\vwt_f(A)=\vwt_f(B)=\dwt_f(\widehat{B})=\pwt_f(B)$.
\end{dfn}

\begin{lem}
  \label{lem:BF}
  Suppose $B$ is a good cover of $A$.  Given $F\subseteq\widehat{B}$
  let us write $A_F=\{\sigma\in A\mid\llb
  F\rrb\supseteq\llb\sigma\rrb\}$ and $B_F=\{\tau\in B\mid\llb
  F\rrb\supseteq\llb\tau\rrb\}$.  Then $B_F$ is a good cover of $A_F$.
\end{lem}
 
\begin{proof}
  Clearly $\widehat{B_F}=F$, hence $\llb A_F\rrb\subseteq\llb
  F\rrb=\llb\widehat{B_F}\rrb=\llb B_F\rrb$.  In order to show that
  $B_F$ is a good cover of $A_F$, it remains to show that
  $\dwt_f(P)\le\dwt_f(S)$ whenever $P\subseteq B_F$ is prefix-free and
  $\llb A_F\rrb\subseteq\llb S\rrb$.  Letting $G=\widehat{B}\setminus
  F$ we see that $P\cap G=\emptyset$ and $P\cup G$ is a prefix-free
  subset of $B$ and $\llb A\rrb=\llb A_F\rrb\cup\llb
  A_G\rrb\subseteq\llb S\rrb\cup\llb G\rrb=\llb S\cup G\rrb$.  Thus
  $\dwt_f(P)+\dwt_f(G)=\dwt_f(P\cup
  G)\le\pwt_f(B)\le\vwt_f(A)\le\dwt_f(S\cup G)\le\dwt_f(S)+\dwt_f(G)$,
  hence $\dwt_f(P)\le\dwt_f(S)$, Q.E.D.
\end{proof}

\begin{rem}
  \label{rem:Btau}
  Let $B$ be a good cover of $A$, and suppose $\tau$ is such that
  $\llb B\rrb\not\supseteq\llb\tau\rrb$.  Then obviously no initial
  segment of $\tau$ belongs to $B$.  In other words,
  $\tau\in\widehat{B\cup\{\tau\}}$.  Letting
  $F=\widehat{B\cup\{\tau\}}\setminus\{\tau\}$ and applying Lemma
  \ref{lem:BF}, we see that $\widehat{B_F}=F$ and $B_F$ is a good
  cover of $A_F$.
\end{rem}

\begin{dfn}
  \label{dfn:wc}
  We define $f:\{0,1\}^*\to[-\infty,\infty]$ to be \underline{convex}
  if $\wt_f(\sigma)\le\wt_f(\sigma\cat\langle0\rangle)
  +\wt_f(\sigma\cat\langle1\rangle)$ for all $\sigma\in\{0,1\}^*$.
  Equivalently, $\wt_f(\sigma)\le\dwt_f(S)$ for all
  $\sigma\in\{0,1\}^*$ and all $S\subseteq\{0,1\}^*$ such that $\llb
  S\rrb=\llb\sigma\rrb$.
\end{dfn}

\begin{lem}
  \label{lem:B'}
  Assume that $f$ is convex.  Suppose $B$ is a good cover of $A$ but
  not of $A'=A\cup\{\sigma\}$.  Choose $\tau\subseteq\sigma$ so as to
  minimize $\dwt_f(\widehat{B\cup\{\tau\}})$.  Then $B'=B\cup\{\tau\}$
  is a good cover of $A'$.
\end{lem}

\begin{proof}
  Obviously $\llb B'\rrb\supseteq\llb A'\rrb$ so it remains to prove
  that $\dwt_f(P')\le\dwt_f(S')$ whenever $P'\subseteq B'$ is
  prefix-free and $\llb A'\rrb\subseteq\llb S'\rrb$.

  Since $\llb\sigma\rrb\subseteq\llb A'\rrb\subseteq\llb S'\rrb$, let
  $\tau^*\subseteq\sigma$ be as short as possible such that
  $\llb\tau^*\rrb\subseteq\llb S'\rrb$.  Obviously $S'$ contains no
  proper initial segment of $\tau^*$.  Hence $\llb\tau^*\rrb=\llb
  S^*\rrb$ for some $S^*\subseteq S'$.  It follows by Definition
  \ref{dfn:wc} that $\wt_f(\tau^*)\le\dwt_f(S^*)$.  Therefore,
  replacing $S'$ by $(S'\setminus S^*)\cup\{\tau^*\}$, we may safely
  assume that $\tau^*\in S'$.

  Since $\llb\sigma\rrb\not\subseteq\llb B\rrb$ and
  $\tau\subseteq\sigma$ and $\tau^*\subseteq\sigma$, we obviously have
  $\llb\tau\rrb\not\subseteq\llb B\rrb$ and
  $\llb\tau^*\rrb\not\subseteq\llb B\rrb$.  Applying Remark
  \ref{rem:Btau} to $\tau$ and to $\tau^*$, we obtain sets
  $F=\widehat{B\cup\{\tau\}}\setminus\{\tau\}$ and
  $F^*=\widehat{B\cup\{\tau^*\}}\setminus\{\tau^*\}$.  In particular,
  since $\llb A\rrb\subseteq\llb S'\rrb$ we have $\llb
  A_{F^*}\rrb=\llb A\rrb\setminus\llb\tau^*\rrb \subseteq\llb
  S'\rrb\setminus\llb\tau^*\rrb\subseteq\llb
  S'\setminus\{\tau^*\}\rrb$.  Moreover, by our choice of $\tau$ we
  have $\dwt_f(\widehat{B\cup\{\tau\}})
  \le\dwt_f(\widehat{B\cup\{\tau^*\}})$.

  We are now ready to complete the proof of Lemma \ref{lem:B'}.  If
  $\tau\notin P'$ we have $P'=P\subseteq B$, hence
  $\dwt_f(P)\le\pwt_f(B)\le\vwt_f(A)\le\dwt_f(S')$ and we are done.
  Suppose now that $\tau\in P'$.  Then $P'=P\cup\{\tau\}$ where
  $P\subseteq B_F$.  Thus we have
  \[
  \begin{array}{rcl}
    \dwt_f(P') &=& \dwt_f(P)+\wt_f(\tau)\\[6pt]
    &\le& \pwt_f(B_F)+\wt_f(\tau)\\[6pt]
    &=& \dwt_f(F)+\wt_f(\tau)\\[6pt]
    &=& \dwt_f(\widehat{B\cup\{\tau\}})\\[6pt]
    &\le& \dwt_f(\widehat{B\cup\{\tau^*\}})\\[6pt]
    &=& \dwt_f(F^*)+\wt_f(\tau^*)\\[6pt]
    &=& \vwt_f(A_{F^*})+\wt_f(\tau^*)\\[6pt]
    &\le& \dwt_f(S'\setminus\{\tau^*\})+\wt_f(\tau^*)\\[6pt]
    &=& \dwt_f(S')
  \end{array}
  \]
  and again we are done.
\end{proof}

\begin{dfn}
  Given $f:\{0,1\}^*\to[-\infty,\infty]$ define
  \begin{center}
    $L_f:\{(P_1,P_2)\mid P_1,P_2$ are finite and
    prefix-free$\}\to\{0,1\}$
  \end{center}
  by
  \[
  L_f(P_1,P_2)=\left\{
    \begin{array}{ll}
      1 & \hbox{if }\dwt_f(P_1)<\dwt_f(P_2),\\[4pt]
      0 & \hbox{otherwise.}
    \end{array}
  \right.
  \]
  We say that $f$ is \underline{strongly computable} if both $f$ and
  $L_f$ are computable.  This is often the case, e.g., if $f$ is
  computable and integer-valued as in Lemma \ref{lem:int}.  Note also
  that Lemma \ref{lem:If} depends only on strong computability.
\end{dfn}

\begin{lem}
  \label{lem:AB}
  Let $f$ be strongly computable and convex.  If $A$ is r.e., we can
  effecively find an r.e.\ set $B$ such that $B$ is a good cover of
  $A$.
\end{lem}

\begin{proof}
  For $n=0,1,2,\ldots$ let $A_n$ consist of the first $n$ elements in
  some fixed recursive enumeration of $A$.  Assume inductively that we
  have found a finite set $B_n$ which is a good cover of $A_n$.  Let
  $A_{n+1}=A_n\cup\{\sigma_n\}$.  If $\llb\sigma_n\rrb\subseteq\llb
  B_n\rrb$ let $B_{n+1}=B_n$.  Otherwise, use strong computability to
  effectively choose $\tau_n\subseteq\sigma_n$ which minimizes
  $\dwt_f(\widehat{B_n\cup\{\tau_n\}})$.  Lemma \ref{lem:B'} then
  implies that $B_{n+1}=B_n\cup\{\tau_n\}$ is a good cover of
  $A_{n+1}$.  Finally let $B=\bigcup_{n=1}^\infty B_n$.  Clearly $B$
  is r.e.\ and $\llb A\rrb\subseteq\llb B\rrb$, so it remains to prove
  that $\dwt_f(P)\le\vwt_f(A)$ for all prefix-free sets $P\subseteq
  B$.  But clearly $\dwt_f(P)=\sup\{\dwt_f(P_0)\mid P_0$ is a finite
  subset of $P\}$, so it suffices to consider finite prefix-free sets.
  If $P\subseteq B$ is finite and prefix-free, let $n$ be such that
  $P\subseteq B_n$.  Then
  $\dwt_f(P)\le\pwt_f(B_n)\le\vwt_f(A_n)\le\vwt_f(A)$, Q.E.D.
\end{proof}

\begin{lem}
  \label{lem:AB2}
  Let $f$ be strongly computable and convex.  If $A$ is r.e., we can
  effecively find an r.e.\ set $B$ such that $\llb A\rrb\subseteq\llb
  B\rrb$ and $\pwt_f(B)\le\vwt_f(A)$.
\end{lem}

\begin{proof}
  This is a restatement of Lemma \ref{lem:AB}.
\end{proof}

\begin{dfn}
  \label{dfn:vfr}
  Assume that $f:\{0,1\}^*\to[-\infty,\infty]$ is computable.  We
  define $X\in\{0,1\}^*$ to be \underline{vehemently $f$-random} if
  $X\notin\bigcap_i\llb A_i\rrb$ whenever $A_i$ is uniformly r.e.\
  such that $\vwt_f(A_i)\le2^{-i}$.
\end{dfn}

\begin{thm}
  \label{thm:veh-str}
  Let $f:\{0,1\}^*\to[-\infty,\infty]$ be strongly computable and
  convex.  Then vehement $f$-randomness is equivalent to strong
  $f$-randomness.
\end{thm}

\begin{proof}
  Suppose $X$ is not strongly $f$-random, say $X\in\bigcap_i\llb
  A_i\rrb$ where $A_i$ is uniformly r.e.\ and $\pwt_f(A_i)\le2^{-i}$.
  By Lemma \ref{lem:vlep} we have $\vwt_f(A_i)\le\pwt_f(A_i)\le2^{-i}$
  so $X$ is not vehemently $f$-random.

  Now suppose $X$ is not vehemently $f$-random, say $X\in\bigcap_i\llb
  A_i\rrb$ where $A_i$ is uniformly r.e.\ and $\vwt_f(A_i)\le2^{-i}$.
  By Lemma \ref{lem:AB2} we can find uniformly r.e.\ $B_i$ such that
  $\llb A_i\rrb\subseteq\llb B_i\rrb$ and
  $\pwt_f(B_i)\le\vwt_f(A_i)\le2^{-i}$.  Clearly $X\in\bigcap_i\llb
  B_i\rrb$, so $X$ is not strongly $f$-random.
\end{proof}

We now sketch how to replace ``strongly computable'' by ``computable.''

\begin{lem}
  \label{lem:sc}
  Let $f$ be computable and convex.  Given $\epsilon>0$ we can
  effectively find an $\overline{f}$ which is strongly computable and
  convex and such that
  $f(\sigma)<\overline{f}(\sigma)<f(\sigma)+\epsilon$ for all
  $\sigma$.
\end{lem}

\begin{proof}
  Let $\QQ$ be the set of rational numbers.  By a straightforward but
  awkward construction, we can find $\overline{f}:\{0,1\}^*\to\QQ$
  which is computable and convex and such that
  $f(\sigma)<\overline{f}(\sigma)<f(\sigma)+\epsilon$ for all
  $\sigma$.  From the $\QQ$-valuedness of $\overline{f}$ it follows
  easily that $\overline{f}$ is strongly computable.
\end{proof}

\begin{lem}
  \label{lem:AB3}
  Let $f$ be computable and convex.  Given $\delta>0$ and an r.e.\ set
  $A$, we can effectively find an r.e.\ set $B$ such that $\llb
  A\rrb\subseteq\llb B\rrb$ and
  $\pwt_f(B)\le(1+\delta)\cdot\vwt_f(A)$.
\end{lem}

\begin{proof}
  Let $\overline{f}$ be as in Lemma \ref{lem:sc} with
  $\epsilon=\log_2(1+\delta)$.  If $A$ is r.e., apply Lemma
  \ref{lem:AB2} to find an r.e.\ set $B$ such that $\llb
  A\rrb\subseteq\llb B\rrb$ and
  $\pwt_{\overline{f}}(B)\le\vwt_{\overline{f}}(A)$.  It is then easy
  to check that $\pwt_f(B)\le(1+\delta)\cdot\vwt_f(A)$.
\end{proof}

\begin{thm}
  \label{thm:veh-str-2}
  Let $f:\{0,1\}^*\to[-\infty,\infty]$ be computable and convex.  Then
  vehement $f$-randomness is equivalent to strong $f$-randomness.
\end{thm}

\begin{proof}
  Proceed as in the proof of Theorem \ref{thm:veh-str} but instead of
  Lemma \ref{lem:AB2} use Lemma \ref{lem:AB3} with $\delta=1$.
\end{proof}

\section{Propagation of vehement $f$-randomness}
\label{sec:vfrPA}

In this section we present an alternative proof of one of our main
results concerning propagation of strong $f$-randomness, Corollary
\ref{cor:sfrPA}.  Our alternative proof proceeds via vehement
$f$-randomness and depends heavily on Remark \ref{rem:vwt}.  Our
alternative proof has the advantage of being a direct generalization
of one of the known proofs (see \cite[Proposition 7.4]{do-hi-mi-ni})
of the corresponding result for Martin-L\"of randomness, Theorem
\ref{thm:mlrQ}.

\begin{thm}
  \label{thm:sfrQ}
  Let $f:\{0,1\}^*\to[-\infty,\infty]$ be computable and convex.  Let
  $Q$ be a nonempty $\Pi^0_1$ subset of $\{0,1\}^\NN$.  If $X$ is
  strongly $f$-random, then $X$ is strongly $f$-random relative to
  some $Z\in Q$.
\end{thm}

\begin{proof}
  Relativizing Corollary \ref{cor:srsc} let $S_i^Z$ be a universal
  test for strong $f$-randomness relative to $Z$.  In other words,
  $S_i^Z$ is uniformly r.e.\ relative to $Z$ and
  $\pwt_f(S_i^Z)\le2^{-i}$ and $\forall X\,\forall
  Z\,(X\notin\bigcap_i\llb S_i^Z\rrb\liff X$ is strongly $f$-random relative
  to $Z)$.  By Lemma \ref{lem:vlep} we have
  $\vwt_f(S_i^Z)\le\pwt_f(S_i^Z)\le2^{-i}$ so by Theorem
  \ref{thm:veh-str-2} $S_i^Z$ is also a universal test for vehement
  $f$-randomness relative to $Z$.  Thus, letting $U_i^Z=\llb S_i^Z\rrb$, we
  have
  \begin{center}
    $\forall X\,\forall Z\,(X\notin\bigcap_iU_i^Z\liff X$ is
    vehemently $f$-random relative to $Z)$
  \end{center}
  and $U_i^Z$ is uniformly $\Sigma^0_1$ relative to $Z$.

  Let $\widetilde{U}_i=\bigcap_{Z\in Q}U_i^Z$.  Since $Q$ is
  $\Pi^0_1$, it follows by compactness that $\widetilde{U}_i$ is
  uniformly $\Sigma^0_1$.  Therefore, let $\widetilde{S}_i$ be
  uniformly r.e.\ such that $\llb\widetilde{S}_i\rrb=\widetilde{U}_i$.  For
  any $Z\in Q$ we have $\widetilde{U}_i\subseteq U_i^Z$, i.e.,
  $\llb\widetilde{S}_i\rrb\subseteq\llb S_i^Z\rrb$, so
  $\vwt_f(\widetilde{S}_i)\le\vwt_f(S_i^Z)\le2^{-i}$ by Remark
  \ref{rem:vwt}.  Thus $\widetilde{S}_i$ is a test for vehement
  $f$-randomness.  In particular we have $\forall X\,(X$ vehemently
  $f$-random $\limp X\notin\bigcap_i\widetilde{U}_i)$.

  Suppose now that $X$ is strongly $f$-random.  By Theorem
  \ref{thm:veh-str-2} $X$ is vehemently $f$-random, so
  \[
  X\notin\bigcap_i\widetilde{U}_i=\bigcap_i\bigcap_{Z\in Q}U_i^Z
  =\bigcap_{Z\in Q}\bigcap_iU_i^Z.
  \]
  Let $Z\in Q$ be such that $X\notin\bigcap_iU_i^Z$.  Then $X$ is
  vehemently $f$-random relative to $Z$, so by Theorem
  \ref{thm:veh-str-2} $X$ is strongly $f$-random relative to $Z$,
  Q.E.D.
\end{proof}

\begin{thm}
  \label{thm:vfrQ}
  Let $f:\{0,1\}^*\to[-\infty,\infty]$ be computable and convex.  Let
  $Q$ be a nonempty $\Pi^0_1$ subset of $\{0,1\}^\NN$.  If $X$ is
  vehemently $f$-random, then $X$ is vehemently $f$-random relative to
  some $Z\in Q$.
\end{thm}

\begin{proof}
  This is immediate from Theorems \ref{thm:veh-str-2} and
  \ref{thm:sfrQ}.
\end{proof}

\section{Other characterizations of strong $f$-randomness}
\label{sec:sfrPA}

In this section we present two new characterizations of strong
$f$-randomness.  One of our new characterizations is in terms of
$f$-randomness relative to a $\PA$-degree.  The other is in terms of
what we call \emph{provable noncomplexity}.

\begin{thm}
  \label{thm:frPA}
  Let $f:\{0,1\}^*\to[-\infty,\infty]$ be computable.  The following
  are pairwise equivalent.
  \begin{enumerate}
  \item $X$ is strongly $f$-random.
  \item $X$ is strongly $f$-random relative to some $\PA$-degree.
  \item $X$ is $f$-random relative to some $\PA$-degree.
  \end{enumerate}
\end{thm}

\begin{proof}
  The implication $1\limp2$ follows from Theorem \ref{thm:sfrPA}.  The
  implication $2\limp3$ is trivial.  It remains to prove $3\limp1$.
  Assume that $1$ fails, i.e., $X$ is not strongly $f$-random.  Let
  $A_i$, $i\in\NN$ be uniformly r.e.\ such that $\pwt_f(A_i)\le2^{-i}$
  and $X\in\bigcap_i\llb A_i\rrb$.  For each $i$ let $\widehat{A}_i$
  be the set of minimal elements of $A_i$.  Let $Q$ be the set of
  sequences $Z_i$, $i\in\NN$ such that $Z_i\subseteq\{0,1\}^*$ and
  $\dwt_f(Z_i)\le2^{-i}$ and $\forall\sigma\,(\sigma\in
  A_i\limp\exists\rho\,(\rho\subseteq\sigma$ and $\rho\in Z_i))$.  The
  sequence $\widehat{A}_i$, $i\in\NN$ belongs to $Q$, so $Q$ is
  nonempty.  Moreover, $Q$ may be viewed as a $\Pi^0_1$ set in the
  Cantor space.  Therefore, given $Z$ of $\PA$-degree, we can find a
  sequence $B_i$, $i\in\NN$ which is Turing reducible to $Z$ and
  belongs to $Q$.  From the definition of $Q$ it follows that
  $\dwt_f(B_i)\le2^{-i}$ and $X\in\bigcap_i\llb B_i\rrb$.  Thus $X$ is
  not $f$-random relative to $Z$.  This holds for all $\PA$-degrees,
  so $3$ fails, Q.E.D.
\end{proof}

For our second characterization, let $\PA$ denote first-order Peano
arithmetic.  Within $\PA$ we define prefix-free complexity
$\KP:\{0,1\}^*\to\NN$ and a priori complexity
$\KA:\{0,1\}^*\to(0,\infty)$ as usual.  Also within $\PA$ we define
$\KP^{(j)}=$ prefix-free complexity relative to $0^{(j)}$, and
$\KA^{(j)}=$ a priori complexity relative to $0^{(j)}$, where
$0^{(j)}=$ the $j$th Turing jump of $0$.  Let $f:\{0,1\}^*\to\NN$ and
$X\in\{0,1\}^\NN$ be arbitrary.

\begin{dfn}
  Let $K$ stand for $\KP$ or $\KA$, and let $Z$ be a Turing oracle.
  We define $X$ to be \underline{$K^Z$-$f$-complex} if $\exists
  c\,\forall n\,(K^Z(X\res n)>f(X\res n)-c))$.
\end{dfn}

\begin{dfn}
  Let $M$ be a nonstandard model of $\PA$.  We define $X$ to be
  \underline{$M$-$f$-complex} if $\forall r\,\exists c\,\forall
  n\,(\KP_r(X\res n)>f(X\res n)-c)$.  Here $r$ ranges over $M$-finite
  functions $r$ with prefix-free domain, and
  $\KP_r(\tau)=\min\{|\sigma|\mid r(\sigma)=\tau\}$.
\end{dfn}

\begin{dfn}
  Let $K$ stand for $\KP$ or $\KA$ or $\KP^{(j)}$ or $\KA^{(j)}$.  Let
  $T$ be a consistent theory extending $\PA$.  We define $X$ to be
  \underline{provably $K$-$f$-noncomplex in $T$} if $\forall
  c\,\exists n\,(T\vdash(\exists m<n)\,(K(X\res m)<f(X\res m)-c))$.
\end{dfn}

\begin{thm}
  \label{thm:TM}
  Let $K$ stand for $\KP$ or $\KA$.  Let $T$ be a recursively
  axiomatizable, consistent theory extending $\PA$.  Let
  $f:\{0,1\}^*\to\NN$ and $X\in\{0,1\}^\NN$ be arbitrary.  For each
  $j\in\NN$ the following are pairwise equivalent.
  \begin{enumerate}
  \item $X$ is not $K^Z$-$f$-complex for any $Z$ of $\PA$-degree.
  \item $X$ is not $M$-$f$-complex for any nonstandard $M\models T$.
  \item $X$ is provably $K$-$f$-noncomplex in some recursively
    axiomatizable, consistent theory extending $T$.
  \item $X$ is provably $K^{(j)}$-$f$-noncomplex in some recursively
    axiomatizable, consistent theory extending $T$.
  \end{enumerate}
\end{thm}

\begin{proof}
  This is a special case of \cite[Theorems 4.1 and
  4.4]{yokoyama-TM-arxiv-131014}.
\end{proof}

\begin{thm}
  \label{thm:sfrTM}
  Let $f:\{0,1\}^*\to\NN$ be computable.  Let $T$ be a recursively
  axiomatizable, consistent theory extending $\PA$.  For all
  $X\in\{0,1\}^\NN$ the following are pairwise equivalent.
  \begin{enumerate}
  \item $X$ is strongly $f$-random.
  \item $X$ is $M$-$f$-complex for some nonstandard $M\models T$.
  \item $X$ is not provably $\KP$-$f$-noncomplex in any recursively
    axiomatizable, consistent theory extending $T$.
  \end{enumerate}
\end{thm}

\begin{proof}
  By Theorems \ref{thm:srsc} and \ref{thm:frPA} $X$ is
  strongly-$f$-random if and only if $X$ is $\KA^Z$-$f$-complex
  for all $Z$ of $\PA$-degree.  The equivalences $1\liff2$ and
  $1\liff3$ then follow by Theorem \ref{thm:TM}.
\end{proof}

Define $K^Z$-length-complexity to mean $K^Z$-$f$-complexity
where $f(\sigma)=$ the length of $\sigma$, and similarly for
$M$-length-complexity and provable $\KP$-length-noncomplexity.

\begin{thm}
  \label{thm:mlrTM}
  Let $T$ be a recursively axiomatizable, consistent theory extending
  $\PA$.  For all $X\in\{0,1\}^\NN$ the following are pairwise
  equivalent.
  \begin{enumerate}
  \item $X$ is Martin-L\"of random.
  \item $X$ is $M$-length-complex for some nonstandard
    $M\models T$.
  \item $X$ is not provably $\KP$-length-noncomplex in any recursively
    axiomatizable, consistent theory extending $T$.
  \end{enumerate}
\end{thm}

\begin{proof}
  This is the special case of Theorem \ref{thm:sfrTM} with
  $f(\sigma)=|\sigma|$ for all $\sigma$.
\end{proof}

\begin{rem}
  \label{rem:TM}
  By Theorem \ref{thm:TM} our notion of provable length-noncomplexity
  is stable under relativation to a strong oracle.  Thus, letting $X$
  be Martin-L\"of random and Turing reducible to $0^{(1)}$, we see
  that $X$ is not $\KP^{(1)}$-length-complex but not provably
  $\KP^{(1)}$-noncomplex in any recursively axiomatizable, consistent
  extension $T$ of $\PA$.  It follows that for any such $T$ there
  exist $\tau\in\{0,1\}^*$ and $n\in\NN$ such that $\KP^{(1)}(\tau)<n$
  but $T\not\vdash\KP^{(1)}(\tau)<n$.  Comparing this to the
  celebrated Chaitin Incompleteness Theorem
  \cite{chaitin-inc,raat-chaitin}, we now have a somewhat different
  example of a statement which is true but not provable in $T$.
\end{rem}

\section{Non-propagation of $f$-randomness}
\label{sec:non-pfr}

In this section we show that Theorems \ref{thm:sfrXYZ} and
\ref{thm:sfrPA} and \ref{thm:sfrTM} fail if strong $f$-randomness is
replaced by $f$-randomness.

\begin{thm}
  \label{thm:non-pfr}
  For many $f$'s, e.g., $f(\sigma)=|\sigma|/2$, we can find an $X$
  which is $f$-random but not $f$-random relative to any $\PA$-degree.
\end{thm}

\begin{proof}
  By Reimann/Stephan \nocite{9asian}\cite{re-st-tests} let $X$ be
  $f$-random but not strongly $f$-random.  By Theorem \ref{thm:frPA}
  $X$ is not $f$-random relative to any $\PA$-degree.
\end{proof}

\begin{cor}
  \label{cor:non-pfr}
  For many $f$'s, e.g., $f(\sigma)=|\sigma|/2$, we can find an $X$
  which is $f$-random but provably $\KP$-$f$-noncomplex in some
  recursively axiomatizable, consistent extension of $\PA$.  Indeed,
  $X$ is provably $\KP$-$f$-noncomplex in some recursively
  axiomatizable, consistent extension of any recursively
  axiomatizable, consistent extension of $\PA$.
\end{cor}

\begin{proof}
  By Theorem \ref{thm:non-pfr} let $X$ be $f$-random but not
  $f$-random relative to any $\PA$-degree.  The desired conclusion
  follows by Theorem \ref{thm:TM}.
\end{proof}

\end{document}